\documentclass{article}
\usepackage[utf8]{inputenc}

\usepackage[margin=2.5cm]{geometry}
\usepackage{graphicx}
\usepackage{amsmath}
\usepackage{amssymb}
\usepackage{amsfonts}
\usepackage{amsthm}
\usepackage{capt-of}
\usepackage{url}
\usepackage[all]{xy}
\usepackage{makeidx}
\usepackage{latexsym}
\usepackage{empheq} 
\usepackage{verbatim}
\usepackage{stmaryrd}
\usepackage{enumitem}
\usepackage{bussproofs}
\usepackage{boxedminipage}
\usepackage{hyperref}
\usepackage{ccaption}
\usepackage{cancel}
\usepackage{multicol}
\usepackage{xcolor}
\usepackage{proof}
\usepackage{longtable}
\usepackage{marginnote}
\newcommand{\bba}{\mathbb{A}}
\newcommand{\bbas}{\mathbb{A}^\delta}
\newcommand{\nomj}{\mathbf{j}}
\newcommand{\nomk}{\mathbf{k}}
\newcommand{\nomi}{\mathbf{i}}
\newcommand{\cnomm}{\mathbf{m}}
\newcommand{\cnomn}{\mathbf{n}}
\newcommand{\cnomo}{\mathbf{o}}
\newcommand{\jty}{J^{\infty}}
\newcommand{\mty}{M^{\infty}}
\newcommand{\adnote}[1]{\textcolor{blue}{ADD:#1}}

\newcommand{\mpnote}[1]{\textcolor{red}{MP:#1}}

\newcommand{\marginkmnote}[1]{\marginnote{\textcolor{green}{KM:#1}}}
\newcommand{\afnote}[1]{\textcolor{orange}{AF:#1}}

\newcommand{\xwnote}[1]{\textcolor{blue}{XW:#1}}

\newcommand{\pcon}{\mathcal{C}}
\newcommand{\ppcon}{\mathcal{D}}
\newcommand{\npcon}{\cancel{\mathcal{C}}}
\newcommand{\nppcon}{\cancel{\mathcal{D}}}
\newcommand{\wt}{{\rhd}}
\newcommand{\bt}{{\blacktriangleright}}
\newcommand{\tw}{{\lhd}}
\newcommand{\tb}{{\blacktriangleleft}}

\def\aol{\rule[0.5865ex]{1.38ex}{0.1ex}}
\def\pdla{\mbox{\rotatebox[origin=c]{180}{$\,>\mkern-8mu\raisebox{-0.065ex}{\aol}\,$}}}

\newtheorem{theorem}{Theorem}[section]
\newtheorem{lemma}[theorem]{Lemma}
\newtheorem{example}[theorem]{Example}
\newtheorem{definition}[theorem]{Definition}
\newtheorem{corollary}[theorem]{Corollary}
\newtheorem{proposition}[theorem]{Proposition}
\newtheorem{remark}[theorem]{Remark}

\usepackage{authblk}

\begin{document}
\title{Obligations and permissions on selfextensional logics}

\author[1]{Andrea De Domenico}
\author[2]{Ali Farjami}
\author[1]{Krishna Manoorkar}
\author[1,3]{Alessandra Palmigiano}
\author[1]{Mattia Panettiere}
\author[1,4]{Xiaolong Wang}

\affil[1]{\small School of Business and Economics, Vrije Universiteit Amsterdam, De Boelelaan 1105, Amsterdam, 1081 HV, The Netherlands}
\affil[2]{\small School of Computer Engineering, Iran University of Science and Technology, Tehran, Iran}
\affil[3]{\small Department of Mathematics and Applied Mathematics, University of Johannesburg, Corner Kingsway and University Road, Rossmore, Johannesburg, South Africa}
\affil[4]{\small School of Philosophy and Social Development, Shandong University, South Shanda Road No.27, Jinan, 250100, China}

\date{}

\maketitle
\begin{abstract}
We further develop the abstract algebraic logic approach to input/output logic initiated  in \cite{wollic22}, where the family of selfextensional logics was proposed as a general  background environment for  input/output logics.
In this paper, we introduce and discuss the generalizations of several types of permission (negative, dual negative, static, dynamic), as well as their interactions with normative systems, to various families of selfextensional logics, thereby proposing a systematic approach to the definition of normative and permission systems  on nonclassical propositional bases.\\
{\em Keywords}: input/output logic, selfextensional logics, abstract algebraic logic.
  %
  \end{abstract}

\section*{Declarations}

\subsection*{Competing interests}
    The authors of this study declare that there is no conflict of interest with any commercial or financial entities related to this research.
\subsection*{Authors' contributions}
    Xiaolong Wang drafted the initial version of this article. Other authors have all made equivalent contributions to it.
\subsection*{Funding}
    The authors who affiliated by Vrije Universiteit has received funding from the European Union’s Horizon 2020 research and innovation programme under the Marie Skłodowska-Curie grant agreement No. 101007627.
\\
Xiaolong Wang is supported by the China Scholarship Council No.202006220087.
\\
Krishna Manoorkar is supported by the Nederlandse Organisatie voor Wetenschappelijk Onderzoek grant KIVI.2019.001 awarded to Alessandra Palmigiano.

\newpage

\section{Introduction}

The present paper continues a line of investigation, recently initiated in \cite{wollic22}, which studies  {\em input/output logics} from an  algebraic logic perspective \cite{font2003survey}.   

The framework of {\em input/output logic} \cite{Makinson00} has been introduced for modelling the interaction between logical inferences and other agency-related notions  such as  conditional obligations, goals, ideals, preferences, actions, and beliefs, in the context of the formalization of normative systems in philosophical logic and AI.
Recently, the original framework of input/output logic, based on classical propositional logic, has been generalized  to incorporate various forms of  {\em nonclassical} reasoning \cite{parent2014intuitionistic,stolpe2015concept}, and these generalizations have   contextually motivated the introduction of algebraic and proof-theoretic methods in the study of input/output logic   \cite{sun2018proof,ciabattoni2023streamlining}.
In the present paper, 
the various notions of {\em permission systems} introduced in \cite{Makinson03}, namely {\em negative permission}, {\em positive static permission} and {\em dynamic permission}, are generalized and studied uniformly in the context of {\em selfextensional logics} \cite{wojcicki2003logic}, both in themselves, and in connection with normative systems. In the same context, the notion of {\em dual negative permission} system (cf.~Section \ref{ssec:dual negative}) is introduced and studied.

Selfextensional logics (cf.~Section \ref{ssec:selfextensional}) form a wide class of logical systems which have been intensely studied in  abstract algebraic logic \cite{Jansana2006conjunction,jansana2007implication}, but have also been studied from a duality-theoretic \cite{jansana2006referential} and proof-theoretic \cite{avron2020} perspective. Selfextensional logics  are the logics for which the {\em weak replacement property} holds: substituting any two interderivable formulas for a variable in any formula gives rise to interderivable formulas. In algebraic terms, selfextensionality is equivalently defined as the property that the interderivability  relation be a congruence of the algebra of formulas. Besides classical propositional logic, well known examples of nonclassical logics which are selfextensional are  intuitionistic \cite{intuitionistic}, bi-intuitionistic \cite{Rauszer1974}, (classical and positive) modal \cite{dunn95,chellas80}, substructural \cite{GJKO}, quantum \cite{Chiara02}, linear \cite{Girard}, intermediate \cite{super}, De Morgan \cite{demorgan}, semi De Morgan \cite{sankappanavar} logics. More in general, all logics whose canonically associated classes of algebras are varieties of normal (resp.~regular, monotone) (distributive) lattice  expansions \cite{conradie2019algorithmic,conradie2020constructive}, and in which the entailment relation is captured by the order of the algebras,\footnote{That is, 
letting $\mathsf{Alg}(\mathcal{L})$ denote the  class of algebras canonically associated with a given logic  $\mathcal{L}$, 
if $\varphi$ and $\psi$ are formulas, then $\varphi\vdash \psi$ iff $h(\varphi)\leq h(\psi)$ for every $A\in \mathsf{Alg}(\mathcal{L})$ and every homomorphism $h: \mathrm{Fm}\to A$.} are selfextensional.\footnote{That is, the weak replacement property holds for the consequence relation associated with the standard (e.g.~Hilbert-style) presentation of each of these logics.} We refer to these logics as normal (resp.~regular, monotone) (D)LE-logics \cite{conradie2019algorithmic}. Choosing selfextensional logics as the background environment  allows for a systematic  and principled generalization of the theory of input/output logic to a large family of nonclassical logics, capturing a wide variety of reasoning forms directly relevant in their interaction with norms. For instance, as is well known, intuitionistic and intermediate logics capture forms of reasoning for which truth is {\em constructive} and is identified with {\em provability} \cite{plato08}, while De Morgan and semi De Morgan logics capture {\em paraconsistent} forms of reasoning, which allow e.g.~to reason about inconsistent information without lapsing into absurdity; linear logic captures reasoning about (different types of) resources, while  it has been argued (cf.~\cite{conradie2020non}) that  non-distributive LE-logics  capture forms of {\em hyperconstructive} reasoning, in which truth is {\em evidential}, and also forms of categorical reasoning. Moreover, the specific abstract algebraic logic approach to selfextentional logics allows to abstract away from certain idiosyncratic features relative e.g.~to the way in which a given logic is presented.

\paragraph{Structure of the paper.} In Section \ref{sec:prelim}, we collect basic definitions and facts about  selfextensional logics and their metalogical properties, as well as normative and permission systems.
In Section \ref{sec:normative on selfextensional}, we build on \cite{wollic22} and generalize normative systems and their associated output operators to the context of selfextensional logics; 
in Section \ref{sec:permission systems}, we introduce, discuss, and study the properties of negative, dual negative, positive static, and positive dynamic permission systems in the context of selfextensional logics. 
We conclude in Section \ref{sec:conclusions}.

 \section{Preliminaries}

\label{sec:prelim}
The present section 
collects preliminaries on  
selfextensional logics (cf.~Section \ref{ssec:selfextensional}),  and on input/output logics (cf.~Section \ref{ssec: normative systems CPL}) based on these. 
\subsection{Selfextensional logics and metalogical properties}\label{ssec:selfextensional}
\paragraph{Logics as consequence relations.} Abstract algebraic logic \cite{font2003survey} takes the notion of logical entailment rather than theoremhood as primary. Consequently, a {\em logic} is defined as a tuple 
$\mathcal{L}= (\mathrm{Fm}, \vdash)$ such that $\mathrm{Fm}$ is the term algebra (in a given algebraic or logical signature) over a set $\mathsf{Prop}$ of atomic propositions, and $\vdash$ is a {\em consequence relation} on $\mathrm{Fm}$, i.e.~$\vdash$ is a relation between sets of formulas and formulas  such that\footnote{In the literature, (cf.~\cite{font2017general})  consequence relations are typically  required to be also closed under substitution of atomic propositions, that is, for any $\Gamma \cup \{\varphi\} \subseteq \mathrm{Fm}$, if $\Gamma \vdash \varphi$ then $\{\sigma(\gamma)\ |\ \gamma \in \Gamma\}\vdash \sigma(\varphi)$, where $\sigma: \mathrm{Fm} \mapsto \mathrm{Fm}$ is an endomorphism. Another common requirement is compactness, that is, whenever  $\Gamma\vdash \varphi$, then $\Gamma'\vdash \varphi$ for some finite subset $\Gamma'\subseteq \Gamma$. }, for all $\Gamma, \Delta\subseteq \mathrm{Fm}$ and all $\varphi\in \mathrm{Fm}$,

\begin{itemize}
\item [(a)] if $\varphi\in \Gamma$ then $\Gamma\vdash \varphi$; 
\item [(b)] if  $\Gamma\vdash \varphi$ and $\Gamma \subseteq \Delta$, then $\Delta\vdash \varphi$; 
\item [(c)] if $\Delta\vdash \varphi$ and $\Gamma\vdash \psi$ for every $\psi\in \Delta$, then $\Gamma\vdash \varphi$.

\end{itemize} 

Clearly,  any such relation $\vdash$ induces a preorder on $\mathrm{Fm}$, which we still denote $\vdash$, 
by restricting to singletons. 
Let ${\equiv} \subseteq\, \mathrm{Fm}\times \mathrm{Fm}$  denote the equivalence relation induced by the preorder $\vdash$;   that is,  the interderivability relation $\equiv$ is defined by $\varphi\equiv \psi$ iff $\varphi\vdash \psi$ and $\psi\vdash \varphi$.
A logic $\mathcal{L}$ is {\em selfextensional} (cf.~\cite{wojcicki2003logic}) if for any $\varphi, \psi\in \mathrm{Fm}$, if $\varphi\equiv \psi$ then $\delta(\varphi/p)\equiv \delta(\psi/p)$ for every $\delta\in \mathrm{Fm}$. 

\paragraph{Consequence relations and closure operators.} Consequence relations (and hence logics defined as indicated above) can be equivalently presented by means of closure operators:\footnote{For any poset $P$, a map $C: P\to P$ is a {\em closure operator} if, for all $x, y\in P$: (a) $x\leq C(x)$; (b) $x\leq y$ implies $C(x)\leq C(y)$; (c) $C(C(x))\leq C(x)$.}
 let $Cn(\Gamma): = \{ \psi \mid \Gamma \vdash \psi \}$ denote the {\em theory} of $\Gamma$ for any $\Gamma\subseteq \mathrm{Fm}$.\footnote{In what follows, we write e.g.~$Cn(\varphi)$ for $Cn(\{\varphi\})$, and $Cn(\Gamma, \varphi)$ for $Cn(\Gamma\cup\{\varphi\})$.} The set $Cn(\varnothing)$ collects the {\em theorems} of $\mathcal{L}$.\footnote{Different consequence relations might have the same set of theorems, one example being the local and the global consequence relations induced by the class of Kripke frames on the language of classical modal logic.} If $Cn(\varnothing) = \mathrm{Fm}$, then $\mathcal{L}$ is {\em inconsistent}.
The assignment $\Gamma\mapsto Cn(\Gamma)$ defines a closure operator $Cn(-): (\mathcal{P}(\mathrm{Fm}), \subseteq )\to (\mathcal{P}(\mathrm{Fm}), \subseteq )$, and conversely, for any such closure operator $C$, the relation $\vdash_C\ \subseteq \mathcal{P}(\mathrm{Fm})\times \mathrm{Fm}$ defined as $\Gamma\vdash_C \varphi$ iff $\varphi\in C(\Gamma)$ is a consequence relation on $\mathrm{Fm}$. Finally, if $\vdash$ is a consequence relation, $\vdash \ = \ \vdash_{C_\vdash}$, where $C_\vdash$ denotes the closure operator associated with $\vdash$, and if $C$ is a closure operator on $(\mathcal{P}(\mathrm{Fm}), \subseteq )$, then $C \ =\  C_{\vdash_C}$.

\paragraph{Metalogical properties.}
\label{ssec:metalogical properties}
 Taking the notion of consequence relation as primary in defining a logical system allows one to abstract away from specific features of the presentation of a logic, and specifically, from any concrete logical signature. However, as is customary in  abstract algebraic logic literature, the familiar logical connectives such as conjunction, disjunction,  and implication can be reintroduced in terms of their behaviour w.r.t.~the consequence relation of the given logic. This gives rise to {\em metalogical properties} of the closure operator associated with the consequence relation of given logics.  In what follows, we collect the best-known  metalogical properties, capturing the abstract behaviour of conjunction, disjunction, and implication (cf.~\cite{font2003survey,font2017general}), but also other less well-known properties e.g.~those which capture the behaviour of co-implication, negation, and co-negation. In particular, we model  the metalogical properties of negation along the lines of the axiomatic hierarchy presented in \cite{almeida09}.

\begin{enumerate}
    \item The {\em conjunction property} ($\wedge_P$) holds for $\mathcal{L}$ if a term $t(x, y)$ (which we denote $x\wedge y$) exists in the language of $\mathcal{L}$   such that $Cn(\varphi\wedge \psi) = Cn(\{\varphi, \psi\})$ for all $\varphi, \psi\in \mathrm{Fm}$. 
    \item The {\em disjunction property} ($\vee_P$) holds for $\mathcal{L}$ if a term $t(x, y)$ (which we denote $x\vee y$) exists in the language of $\mathcal{L}$   such that $Cn(\varphi\vee \psi) = Cn(\varphi)\cap Cn( \psi)$ for all $\varphi, \psi\in \mathrm{Fm}$. 
    \item The {\em strong disjunction property} ($\vee_S$) holds for $\mathcal{L}$ if a term $t(x, y)$ (which we denote $x\vee y$) exists in the language of $\mathcal{L}$   such that $Cn(\Gamma, \varphi\vee \psi) = Cn(\Gamma,\varphi)\cap Cn(\Gamma, \psi)$ for all $\varphi, \psi\in \mathrm{Fm}$ and every $\Gamma\subseteq \mathrm{Fm}$. 
    \item The {\em bottom property} ($\bot_P$) holds for $\mathcal{L}$ if a term $t$ (which we denote $\bot$) exists in the language of $\mathcal{L}$ such that $Cn(\bot) = \mathrm{Fm}$.
    \item The {\em weak top property} ($\top_W$) holds for $\mathcal{L}$ if a term $t$ (which we denote $\top$) exists in the language of $\mathcal{L}$ such that $\top\in Cn(\varphi)$ for every $\varphi\in \mathrm{Fm}$.
    \item The {\em top property} ($\top_P$) holds for $\mathcal{L}$ if a term $t$ (which we denote $\top$) exists in the language of $\mathcal{L}$ such that $Cn(\top) = Cn(\varnothing)$.
    \item The {\em weak negation property} ($\neg_W$) holds for $\mathcal{L}$ if a term $t(x)$ (which we denote $\neg x$) exists in the language of $\mathcal{L}$   such that $\psi \in Cn(\varphi)$ implies $\neg \varphi \in Cn(\neg \psi)$ for any $\varphi, \psi \in \mathrm{Fm}$. For any logic $\mathcal{L}$ with $\neg_W$,

    \begin{enumerate}
    \item The {\em right-involutive negation property} ($\neg_{Ir}$) holds for $\mathcal{L}$ if   $Cn(\neg\neg \varphi) \subseteq Cn( \varphi)$ for any $\varphi\in \mathrm{Fm}$.
    \item The {\em left-involutive negation property} ($\neg_{Il}$) holds for $\mathcal{L}$ if   $Cn( \varphi) \subseteq Cn(\neg\neg \varphi)$ for any $\varphi\in \mathrm{Fm}$.
    \item The {\em involutive negation property} ($\neg_I$) holds for $\mathcal{L}$ if both $\neg_{Il}$ and $\neg_{Ir}$ hold for $\mathcal{L}$.
    \item The {\em absurd negation property} ($\neg_A$) holds for $\mathcal{L}$ if    $Cn( \varphi, \neg \varphi) = \mathrm{Fm}$ for any $\varphi\in \mathrm{Fm}$.
    \item The {\em pseudo negation property} ($\neg_P$) holds for $\mathcal{L}$ if  $\wedge_P$ holds for $\mathcal{L}$, and moreover,   $\neg \psi \in Cn(\varphi, \neg(\varphi \wedge \psi))$ for any $\varphi, \psi \in \mathrm{Fm}$.
    \item The {\em excluded middle property} ($\sim_A$) holds for $\mathcal{L}$ if    $Cn( \varphi) \cap Cn(\neg \varphi) = Cn(\varnothing)$ for any $\varphi\in \mathrm{Fm}$.
    \item The {\em pseudo co-negation property} ($\sim_P$) holds for $\mathcal{L}$ if  $\vee_P$ holds for $\mathcal{L}$, and moreover,   $\varphi \vee \neg (\varphi \vee \psi) \in Cn(\neg \psi)$ for all $\varphi, \psi \in \mathrm{Fm}$. 
    \item The {\em strong negation property} ($\neg_S$) holds for $\mathcal{L}$ if $Cn(\varphi, \psi) = \mathrm{Fm}$ implies $\neg \psi \in Cn(\varphi)$.
    \end{enumerate}
    \item The (weak)\footnote{We refer to this property as weak, because in the literature the property referred to as {\em deduction-detachment property} is $\psi \in Cn(\Gamma, \varphi)$ iff $\varphi \rightarrow \psi \in Cn(\Gamma)$ for all $\varphi, \psi \in \mathrm{Fm}$ and $\Gamma \subseteq \mathrm{Fm}$.} {\em deduction-detachment property} ($\rightarrow_P$) holds for $\mathcal{L}$ if a term $t(x,y)$ (which we denote $x\rightarrow y$) exists in the language of $\mathcal{L}$ such that $\psi \in Cn(\chi,\varphi)$ iff $\varphi \rightarrow \psi \in Cn(\chi)$ for all $\varphi, \psi \in \mathrm{Fm}$.

    \item The {\em co-implication property} ($\pdla_P$) holds for $\mathcal{L}$ if a term $t(x,y)$ (which we denote $x \pdla y$, to be read as ``$x$ excludes $y$'') exists in the language of $\mathcal{L}$ such that $\chi \in Cn(\varphi \pdla \psi)$ iff $Cn(\chi) \cap Cn(\psi) \subseteq Cn(\varphi)$ for all $\varphi,\psi, \chi \in \mathrm{Fm}$.
\end{enumerate}

\begin{lemma}
\label{lemma:antitonicity of neg}
For any logic $\mathcal{L} = (\mathrm{Fm}, \vdash)$,
\begin{enumerate}
    \item If  properties $\wedge_P$, $\vee_P$, and $\neg_W$ hold for $\mathcal{L}$, then $\neg\varphi\vee \neg \psi\vdash \neg (\varphi\wedge \psi)$ for all $\varphi, \psi\in \mathrm{Fm}$.
    \item If in addition  property $\neg_{Il}$ holds for $\mathcal{L}$, then $ \neg (\varphi\wedge \psi)\vdash \neg\varphi\vee \neg \psi$ for all $\varphi, \psi\in \mathrm{Fm}$.
    \end{enumerate}
\end{lemma}
\begin{proof}
   1.  Let $\varphi, \psi\in \mathrm{Fm}$. By $\neg_W$ and $\wedge_P$, from $\varphi \in Cn(\varphi,\psi) = Cn(\varphi\wedge\psi)$ it follows $\neg(\varphi\wedge\psi) \in Cn(\neg \varphi)$.  Similarly, one shows that $\neg(\varphi\wedge\psi) \in Cn(\neg \psi)$. Hence, by $\vee_P$, $\neg(\varphi \wedge \psi) \in Cn(\neg\varphi) \cap Cn(\neg\psi) = Cn(\neg\varphi\vee\neg\psi)$ holds, from which $\neg\varphi\vee \neg \psi\vdash \neg (\varphi\wedge \psi)$ immediately follows.\\
    2. Since $\neg \varphi \vee \neg \psi \in Cn(\neg \varphi \vee \neg \psi) = Cn(\neg\varphi) \cap Cn(\neg \psi) \subseteq Cn(\neg \varphi)$, by applying $\neg_W$ we get $\neg\neg\varphi \in Cn(\neg(\neg\varphi\vee\neg\psi))$, i.e.~$Cn(\neg\neg\varphi) \subseteq Cn(\neg(\neg\varphi\vee\neg\psi))$, and by  $\neg_{Il}$, $\varphi \in Cn(\neg\neg\varphi) \subseteq Cn(\neg(\neg\varphi\vee\neg\psi))$. Similarly, one shows $\psi \in Cn(\neg(\neg\varphi\vee\neg\psi))$. The two statements  imply, by $\wedge_P$, that $\varphi \wedge \psi \in Cn(\varphi \wedge \psi) = Cn(\varphi, \psi) \subseteq Cn(\neg(\neg\varphi\vee\neg\psi))$. Applying again $\neg_W$ we obtain $\neg\neg(\neg\varphi\vee\neg\psi) \in Cn(\neg(\varphi\wedge\psi))$ and using $\neg_{Il}$ again we get $\neg\varphi\vee\neg\psi \in Cn(\neg(\varphi\wedge\psi))$, as required.
\end{proof}

\begin{lemma}
\label{lemma:antitonicity of neg second}
For any logic $\mathcal{L} = (\mathrm{Fm}, \vdash)$,
\begin{enumerate}
    \item If  properties $\wedge_P$, $\vee_P$, and $\neg_W$ hold for $\mathcal{L}$, then $\neg(\varphi \vee \psi) \vdash \neg \varphi \wedge \neg \psi$ for all $\varphi, \psi\in \mathrm{Fm}$.
    \item If in addition  property $\neg_S$, $\neg_A$ and $\vee_S$ hold for $\mathcal{L}$, then $\neg \varphi \wedge \neg \psi \vdash \neg(\varphi \vee \psi)$ for all $\varphi, \psi\in \mathrm{Fm}$\footnote{Notice that this can also be proven using $\neg_I$ and the previous Lemma in place of $\vee_S$, $\neg_S$, and $\neg_A$.}.
    \end{enumerate}
\end{lemma} 
\begin{proof}
   1.  Let $\varphi, \psi\in \mathrm{Fm}$. From $\vee_P$ it follows that $\varphi \vee \psi \in Cn(\varphi)$, which implies, by $\neg_W$, that $\neg \varphi \in Cn(\neg(\varphi \vee \psi))$. Similarly, $\neg\psi \in Cn(\neg(\varphi\vee\psi))$. Hence, $Cn(\neg\varphi\wedge\neg\psi) = Cn(\neg\varphi,\neg\psi)   \subseteq Cn(\neg(\varphi\vee\psi))$, as required.\\
    2. 
    Let $\varphi,\psi\in \mathrm{Fm}$. By $\vee_S$ and $\neg_A$ it follows that $Cn(\neg\varphi\wedge\neg\psi,\varphi\vee\psi) = Cn(\neg\varphi\wedge\neg\psi,\varphi) \cap Cn(\neg\varphi\wedge\neg\psi,\psi) \supseteq Cn(\neg\varphi,\varphi) \cap Cn(\neg\psi,\psi) = \mathrm{Fm} \cap \mathrm{Fm} = \mathrm{Fm}$. By  $\neg_S$, this implies $\neg \varphi \wedge \neg \psi \vdash \neg(\varphi \vee \psi)$, as required.
\end{proof}

\begin{proposition} \label{prop:preliminary_properties} For any logic $\mathcal{L} = (\mathrm{Fm}, \vdash)$,
    \begin{enumerate}
    \item if $\wedge_P$ and $\vee_S$ hold for $\mathcal{L}$, then  $\alpha \wedge (\beta \vee \gamma) \vdash (\alpha \wedge \beta) \vee (\alpha \wedge \gamma)$ for all $\alpha, \beta, \gamma\in \mathrm{Fm}$.
    \item For all $\alpha, \beta \in \mathrm{Fm}$, if $\alpha \in Cn(\beta)$, then $Cn(\Gamma, \alpha) \subseteq Cn(\Gamma, \beta)$ for  every $\Gamma\subseteq\mathrm{Fm}$. \label{prop: se_p_two}
\item The following are equivalent:\label{prop: galois}
\begin{enumerate}
\item Property $\neg_{Ir}$ holds of $\mathcal{L}$;
\item for all $\varphi, \psi\in \mathrm{Fm}$, $\varphi \vdash \neg \psi$ iff $\psi \vdash \neg \varphi$. 
\end{enumerate}
\item Property $\top_P$ implies $\top_W$, and if $Cn(\varnothing)\neq \varnothing$, then $\top_W$ implies $\top_P$.
\item \label{item: bot} Properties $\bot_P$ and $\neg_{Ir}$  imply $\top_W$, and in the presence of $Cn(\varnothing)\neq \varnothing$, also $\top_P$.  
\item If $Cn(\varnothing)\neq \varnothing$, and $\wedge_P$, and $\bot_P$ hold, then the following are equivalent:
\begin{enumerate}
\item properties $\neg_{Ir}$, $\neg_A$, and $\neg_P$ hold;
\item $\varphi \wedge \psi \vdash \bot$ iff $\varphi \vdash \neg \psi$ for all $\varphi, \psi\in \mathrm{Fm}$.
\end{enumerate}
\label{prop:intuitionistic}
\item In the presence of $\wedge_P$ and $\vee_S$,  properties $\neg_{Il}$ and $\neg_A$ imply $\neg_P$.
\item In the presence of $\vee_P$, the following are equivalent:
\begin{enumerate}\item property $\pdla_P$ holds;
\item $ \varphi \pdla \psi \vdash \chi$ iff $\varphi \vdash \chi \vee \psi$.
\end{enumerate}
\item If $Cn(\varnothing)\neq \varnothing$, then  $\bot_P$,  $\neg_{Ir}$, and $\neg_P$ imply $\neg_S$. \label{prop:strong_negation}
\item In the presence of $\wedge_P$, the following are equivalent:
\begin{enumerate}\item property $\rightarrow_P$ holds;
\item $\chi\vdash  \varphi \rightarrow \psi$ iff $\varphi\wedge \chi \vdash  \psi$.
\end{enumerate}
\item If $\pdla_P$ holds, then, for all $\alpha, \varphi, \psi \in \mathrm{Fm}$, if $\varphi \vdash \psi$  then $\alpha \pdla \psi \vdash \alpha \pdla \varphi$ and $\varphi\pdla \alpha\vdash \psi\pdla\alpha$. \label{lemma:pdla_antitone}
\item If $\rightarrow_P$ holds, then, for all $\alpha, \varphi, \psi \in \mathrm{Fm}$, if $\varphi \vdash \psi$  then $\alpha \rightarrow \varphi \vdash \alpha \rightarrow \psi $ and $\psi\rightarrow \alpha\vdash \varphi\rightarrow\alpha$. \label{lemma:ra_antitone}
\end{enumerate}
\end{proposition}

\begin{proof}
\begin{enumerate}
    \item  $Cn(\alpha \wedge (\beta \vee \gamma)) = Cn(\alpha, \beta \vee \gamma) = Cn(\alpha, \beta) \cap Cn(\alpha, \gamma) = Cn(\alpha \wedge \beta) \cap Cn(\alpha \wedge \gamma) = Cn((\alpha \wedge \beta) \vee (\alpha \wedge \gamma))$.
    \item By assumption and the monotonicity of $Cn(-)$, $\alpha\in Cn(\beta)\subseteq Cn(\Gamma, \beta)$; moreover, $\Gamma\subseteq Cn(\Gamma)\subseteq Cn(\Gamma, \beta)$. Hence, $Cn(\Gamma, \alpha )\subseteq Cn(\Gamma, \beta)$, as required.
    \item $((a) \Rightarrow (b))$ Without loss of generality we only prove that $\varphi \vdash \neg \psi$ implies $\psi \vdash \neg \varphi$. From $\neg \psi \in Cn(\varphi)$ we get $\neg \varphi \in Cn(\neg \neg \psi) \subseteq Cn(\psi)$, via $\neg_W$ and then $\neg_{Ir}$. This proves the assertion.

    $((b) \Rightarrow (a))$ We need to prove that $Cn(\neg \neg \varphi) \subseteq Cn(\varphi)$, and that $\varphi \in Cn(\psi)$ implies $\neg \psi \in Cn(\neg\varphi)$. For the first part, $\neg \varphi \vdash \neg \varphi$ implies $\varphi\vdash \neg \neg \varphi $, which yields the required inclusion. For the second part, let $\varphi \in Cn(\psi)$. By assumption, to prove $\neg \psi \in Cn(\neg\varphi)$ it is enough to show that $\neg \neg \varphi \in Cn(\psi)$. The last statement holds because $Cn(\neg \neg \varphi) \subseteq Cn(\varphi) \subseteq Cn(\psi)$.
    \item By assumption, $\top\in Cn(\varnothing)\subseteq Cn(\varphi)$ for every $\varphi\in \mathrm{Fm}$, which proves the first part of the statement. For the second part, $Cn(\top)\subseteq Cn(\varphi)\subseteq Cn(\varnothing)$ for any $\varphi\in Cn(\varnothing)$, and by assumption such a $\varphi$ exists.
    \item Let  $\top: = \neg\bot$ be the required term. For any $\varphi\in \mathrm{Fm}$, $\bot\vdash \varphi$ by $\bot_P$, which implies $\neg\varphi\vdash \neg \bot$ by $\neg_W$ for any $\varphi\in \mathrm{Fm}$. Hence in particular, instantiating $\varphi: =\neg\varphi$, we get $\neg\neg \varphi\vdash \neg \bot$ for any $\varphi\in \mathrm{Fm}$, which implies, by $\neg_{Ir}$, that  $\varphi\vdash \neg\neg \varphi\vdash \neg \bot$ for any $\varphi\in \mathrm{Fm}$, which proves the first part of the statement. Specializing the last statement to any $\varphi\in Cn(\varnothing)$ yields $Cn(\neg\bot)\subseteq Cn(\varphi)\subseteq  Cn(\varnothing)$, which completes the proof. 
    \item $((a) \Rightarrow (b))$ If $\varphi \vdash \neg \psi$, then by $\neg_A$ and item \ref{prop: se_p_two} above, $\mathrm{Fm} = Cn(\psi, \neg\psi)\subseteq Cn(\psi, \varphi) = Cn(\varphi\wedge\psi)$, hence $Cn(\varphi \wedge \psi) = Cn(\bot)$, as required.
    Conversely, assume that  $Cn(\bot) \subseteq Cn(\varphi \wedge \psi)$, which implies, by $\neg_W$, that $Cn(\neg(\varphi\wedge\psi)) \subseteq Cn(\neg\bot)$.   
      From this and item \ref{prop: se_p_two} we get $Cn(\varphi, \neg(\varphi\wedge\psi)) \subseteq Cn(\varphi, \neg\bot) = Cn(\varphi)$, the last identity holding since, by  item \ref{item: bot}, $Cn(\neg \bot) = Cn(\varnothing)$. The required statement follows from this inclusion and $\neg_P$.

    $((b) \Rightarrow (a))$ By  item \ref{prop: galois}, to show  $\neg_{Ir}$, it is enough to show that equivalence \ref{prop: galois}(b) holds. If $\varphi\vdash \neg\psi$, then by assumption $\varphi\wedge \psi\vdash \bot$, i.e.~(thanks to $\wedge_P$) $\psi\wedge \varphi\vdash \bot$ iff $\psi\vdash \neg \varphi$, as required. As to $\neg_A$, from $\neg\psi\vdash\neg\psi$ we get $\neg\psi\wedge \psi\vdash \bot$, which, in the presence of $\bot_P$ and $\wedge_P$, is equivalent to $Cn(\neg\psi, \psi) = \mathrm{Fm}$, as required. As to   $\neg_P$, by assumption, the required entailment $\varphi \wedge \neg(\varphi \wedge \psi) \vdash \neg \psi$ is equivalent to $(\varphi \wedge \neg(\varphi \wedge \psi)) \wedge \psi \vdash \bot$, which by $\wedge_P$ is equivalent to $((\varphi \wedge \psi)\wedge \neg(\varphi \wedge \psi))   \vdash \bot$, which is equivalent to $\neg(\varphi \wedge \psi)   \vdash \neg(\varphi \wedge \psi)$, which is trivially true.
    \item  Lemma \ref{lemma:antitonicity of neg} and the assumptions ($\neg_{Il}$ in particular) imply that $Cn(\neg(\varphi \wedge \psi)) = Cn(\neg \varphi \vee \neg \psi)$. Hence, from this, $\vee_S$, $\neg_A$, and item \ref{prop: se_p_two},  $Cn(\varphi, \neg(\varphi \wedge \psi)) = Cn(\varphi, \neg \varphi \vee \neg \psi) = Cn(\varphi, \neg \varphi) \cap Cn(\varphi, \neg \psi) = Cn(\varphi, \neg \psi)$. The required statement follows from this and $\neg \psi \in Cn(\varphi, \neg \psi)$.
    \item ($(a)\Rightarrow(b)$) It is enough to show that $\chi \in Cn(\varphi \pdla \psi)$ iff $ \chi \vee \psi  \in Cn(\varphi)$. By $\pdla_P$ and $\vee_P$,  $\chi \in Cn(\varphi \pdla \psi) $ iff $\chi\vee \psi\in Cn(\chi \vee \psi) = Cn(\chi) \cap Cn(\psi) \subseteq Cn(\varphi)$, as required.

    ($(b)\Rightarrow(a)$) $\chi\in Cn(\varphi\pdla \psi)$ iff $\varphi\pdla \psi\vdash \chi$ iff $\varphi\vdash \chi\vee \psi$ iff $ Cn(\chi) \cap Cn(\psi)  = Cn(\chi \vee \psi) \subseteq Cn(\varphi)$.
    \item By assumption, $Cn (\varphi\wedge \psi) = Cn(\varphi, \psi) = \mathrm{Fm} = Cn(\bot)$, i.e.~$\varphi\wedge \psi\vdash \bot$, which implies $\neg\bot\vdash \neg(\varphi\wedge \psi)$. 
    From this and item \ref{prop: se_p_two} we get $Cn(\varphi, \neg(\varphi\wedge\psi)) \subseteq Cn(\varphi, \neg\bot) = Cn(\varphi)$, the last identity holding since, by  item \ref{item: bot}, $Cn(\neg \bot) = Cn(\varnothing)$. The required statement follows from this inclusion and $\neg_P$.
    \item ($(a)\Rightarrow(b)$) It is enough to show that $\varphi \rightarrow \psi \in Cn(\chi)$ iff $ \psi  \in Cn(\varphi\wedge\chi)$. By $\rightarrow_P$ and $\wedge_P$,  $\varphi \rightarrow \psi \in Cn(\chi)$  iff $ \psi\in  Cn(\varphi, \chi) = Cn(\varphi\wedge \chi)$, as required.

    ($(b)\Rightarrow(a)$) $\varphi\rightarrow \psi\in Cn(\chi)$ iff $\chi \vdash \varphi\rightarrow \psi$ iff $\varphi\wedge \chi\vdash \psi$ iff $  Cn(\psi)  \subseteq Cn(\varphi \wedge \chi)  =  Cn(\varphi, \chi)$.

    \item By $\alpha \pdla \varphi \in Cn(\alpha \pdla \varphi)$  and  $\pdla_P$ we deduce $Cn(\alpha \pdla \varphi) \cap Cn(\varphi) \subseteq Cn(\alpha)$. Hence, from $Cn(\psi) \subseteq Cn(\varphi)$ it follows $Cn(\alpha \pdla \varphi) \cap Cn(\psi) \subseteq Cn(\alpha \pdla \varphi) \cap Cn(\varphi)  \subseteq Cn(\alpha)$, and using $\pdla_P$ again, we conclude  $\alpha \pdla \varphi \in Cn(\alpha \pdla \psi)$, as required. For the second part of the statement, by $\psi \pdla \alpha \in Cn(\psi \pdla \alpha)$  and  $\pdla_P$ we deduce $Cn(\psi \pdla \alpha) \cap Cn(\alpha) \subseteq Cn(\psi)$. This and $Cn(\psi) \subseteq Cn(\varphi)$ imply $Cn(\psi \pdla \alpha) \cap Cn(\alpha)   \subseteq Cn(\varphi)$, and using $\pdla_P$ again, we get  $ \psi\pdla \alpha \in Cn(\varphi\pdla \alpha)$, as required.
    \item The first part of the statement is equivalent to $\alpha\rightarrow \psi\in Cn(\alpha\rightarrow \varphi)$, which, by $\rightarrow_P$, is equivalent to $\psi\in Cn(\alpha\rightarrow \varphi, \alpha)$. Since by assumption $\psi\in Cn(\varphi)$, it is enough to show that $\varphi\in Cn(\alpha\rightarrow \varphi, \alpha)$, which, by $\rightarrow_P$, is equivalent to $\alpha\rightarrow \varphi\in Cn(\alpha\rightarrow \varphi)$, which is true. The second part is equivalent to $\varphi\rightarrow \alpha\in Cn(\psi\rightarrow \alpha)$, which, by $\rightarrow_P$, is equivalent to $\alpha\in Cn(\psi\rightarrow \alpha,\varphi)$. Since by assumption $\psi\in Cn(\varphi)$, by item \ref{prop: se_p_two}, it is enough to show that $\alpha\in Cn(\psi\rightarrow \alpha, \psi)$, which again by $\rightarrow_P$ is equivalent to $\psi\rightarrow \alpha\in Cn(\psi\rightarrow \alpha)$, which is true.
\end{enumerate}
\end{proof}

\begin{example}
\label{ex:main example}
Here we list some well known logics whose standard consequence relations are selfextensional; for each of them, we  highlight the metalogical properties it enjoys, and we also specify a  nontrivial\footnote{For instance, 
a trivial choice of $t(x)$ for which $\neg_A$, $\neg_{Il}$, and $\neg_P$ hold is $t(x)\coloneqq \bot$, and dually, a trivial choice of $t(x)$ for which $\sim_A$, $\neg_{Ir}$, and $\sim_P$ hold is $t(x)\coloneqq \top$.} `term-connective'  witnessing the property whenever it does not belong to the primitive signature with which the given logic is most commonly presented.
\begin{enumerate}  
\item For the selfextensional logic $\mathcal{L}$ canonically associated with the class of lattices with bottom and without top, properties $\wedge_P$, $\vee_P$ and $\bot_P$ hold but  $\top_W$ does not. For this logic, $Cn(\varnothing) = \varnothing$ (i.e.~$\mathcal{L}$ is a logic without theorems).
\item  For  positive modal logic \cite{dunn95}, properties $\wedge_P$ and $\vee_S$, $\bot_P$, $\top_P$ hold and no other property listed above.
\item For orthologic \cite{goldblatt74}, only properties $\wedge_P$, $\vee_P$, $\bot_P$, $\top_P$, $\neg_I$, and $\neg_A$ hold.
\item \label{item:pseudocomplemented} For the logic of pseudocomplemented lattices \cite[Chapter 7]{blyth}, properties $\top_P$, $\bot_P$, $\wedge_P$ $\vee_P$, $\neg_A$ and $\neg_S$ hold. 
\item For the basic Lambek calculus \cite{GJKO},  
property $\neg_W$ holds for $t_1(x): = x\backslash \bot$ and for $t_2(x): = \bot\slash x$.
\begin{enumerate}
   
\item If the `contraction'\footnote{We refer to this axiom as contraction since it corresponds to the well known contraction rule \infer{\varphi\vdash \delta}{\varphi,\varphi\vdash \delta}. Likewise, the weakening axiom corresponds to the weakening rule \infer{\varphi, \psi\vdash \delta}{\varphi\vdash \delta}.} axiom $\varphi \vdash \varphi \otimes \varphi$ is added, then  $\neg_A$ also holds for $t_1(x)$ and $t_2(x)$ as above.
\item If the `weakening' axiom $\varphi\otimes \psi \vdash \varphi$ is added, then $\top_W$ also holds for $t_1(x): =x\backslash x$ and $t_2(x): = x\slash x$. Notice, however, that since $Cn(\varnothing) = \varnothing$,  $\top_P$ does not hold.
\end{enumerate}

\item For semi-De Morgan logic \cite{sankappanavar}, only properties $\wedge_P$, $\vee_S$, $\bot_P$, $\top_P$ and $\neg_W$ hold. 
\begin{enumerate}
\item For lower quasi-De Morgan logic, property $\neg_{Ir}$ also holds.
\item For upper quasi-De Morgan logic, property $\neg_{Il}$ also holds.
\item For almost pseudocomplemented logic, property $\neg_A$ also holds.
\end{enumerate}
Properties $\neg_I$ and $\sim_A$ do not hold for any of these logics. 
\item For De Morgan logic \cite{demorgan}, all the properties of semi-De Morgan logic hold, with the addition of $\neg_I$. However, neither $\neg_A$ nor $\sim_A$ hold.
\item Properties $\wedge_P$, $\vee_S$, $\bot_P$, $\top_P$, and $\rightarrow_P$ hold for intuitionistic logic. Furthermore, properties $\neg_{Ir}$, $\neg_A$, $\neg_P$, and $\neg_S$ hold for $t(x) \coloneqq x \rightarrow \bot$.
\item For bi-intuitionistic logic \cite{Rauszer1974}, all the properties of intuitionistic logic hold plus $\pdla_P$. Furthermore, $\neg_{Il}$, $\sim_A$ and $\sim_P$ hold for $t(x) \coloneqq \top \pdla x$.
\item For any logic based on classical logic (e.g.~classical modal logic $K$ and other modal expansions of CPL), properties $\neg_I$, $\neg_A$, $\neg_P$, $\neg_S$, $\sim_A$, and $\sim_P$ hold, while  $\rightarrow_P$ holds for $t_1(x,y) \coloneqq \neg x \vee y$ and  $\pdla_P$ holds for $t_2(x,y) \coloneqq x \wedge \neg y$.

\item For the implicative fragment of intuitionistic logic, only $\rightarrow_P$ holds (cf.~\cite{Jansana2006conjunction, jansana2007implication}).
  
\end{enumerate}
\end{example}
\subsection{Normative  systems}
\label{ssec: normative systems CPL}

Input/output logic  \cite{Makinson00} is a framework modelling the interaction between the relation of logical entailment between states of affair (states of affair being represented by formulas) and other binary relations on states of affair, representing e.g.~systems of norms, strategies, preferences, and so on.  

Let $\mathcal{L}_{\mathrm{CPL}}= (\mathrm{Fm}, \vdash_{\mathrm{CPL}})$, s.t.~$\mathrm{Fm}$ is the  language of classical propositional logic (CPL) over a given (denumerable) set $\mathsf{Prop}$ of proposition variables, and $\vdash_{\mathrm{CPL}}\, \subseteq\, \mathcal{P}(\mathrm{Fm})\times \mathrm{Fm}$ is the entailment relation of classical propositional logic. Throughout this paper,  {\em formulas}, i.e.~elements in $\mathrm{Fm}$ will be denoted by lowercase Greek letters, and sets of formulas by uppercase Greek letters. For any $\Gamma\subseteq \mathrm{Fm}$, let $Cn_{\mathrm{CPL}}(\Gamma): = \{\varphi\in \mathrm{Fm}\mid \Gamma\vdash_{\mathrm{CPL}} \varphi\}$. 
A  {\em normative system} on $\mathcal{L}_{\mathrm{CPL}}$ is a relation $N \subseteq \mathrm{Fm}\times \mathrm{Fm}$, the elements  $(\alpha,\varphi)$ of which are called {\em conditional norms} (or obligations). 

  A normative system  $N \subseteq \mathrm{Fm}\times\mathrm{Fm}$ is {\em internally incoherent} if  $(\alpha, \varphi)$ and $ (\alpha, \neg \varphi) \in N$ for some $\alpha, \varphi\in \mathrm{Fm}$; a normative system $N$ is {\em internally coherent} if it is not internally incoherent.
  If  $N, N' \subseteq \mathrm{Fm}\times\mathrm{Fm}$ are  normative systems,  $N$ is {\em almost included} in $N'$ (in symbols: $N\subseteq_c N'$) if $(\alpha,\varphi) \in N$ and $\alpha\not\vdash_{\mathrm{CPL}}\bot$ imply $(\alpha,\varphi) \in N'$.

Each  norm $(\alpha,\varphi)\in N$ can be intuitively read as ``given $\alpha$, it {\em should} be the case that $\varphi$''. 
This interpretation can be further specified according to the context: for instance, if $N$ formally represents a system of (real-life) rules/norms, then  we can read $(\alpha,\varphi)\in N$ as ``$\varphi$ is obligatory whenever $\alpha$ is the case''; if $N$ formally represents a scientific theory, then  we can read $(\alpha,\varphi)\in N$ as ``under conditions $\alpha$, one should observe $\varphi$'', in the sense that the scientific theory predicts $\varphi$ whenever $\alpha$; finally, if $N$ formally represents (the execution of) a program, then we can read $(\alpha,\varphi)\in N$ as ``in every state of computation in which $\alpha$ holds, the program will move to a state in which $\varphi$ holds''.  
For any  $\Gamma\subseteq \mathrm{Fm}$, let $N(\Gamma) := \{ \psi \mid \exists \alpha(\alpha \in \Gamma \ \&\ (\alpha,\psi)\in N )\}$. 

    An {\em input/output logic} is a tuple $\mathbb{L} = (\mathcal{L}_{\mathrm{CPL}}, N)$ s.t.~$\mathcal{L}_{\mathrm{CPL}}$ is a classical propositional logic, and $N$ is a normative system on $\mathcal{L}_{\mathrm{CPL}}$.


For any input/output logic $\mathbb{L} = (\mathcal{L}, N)$, and each $1\leq i\leq 4$, the output operation $out_i^N$ is defined as follows: for any $\Gamma\subseteq \mathrm{Fm}$,

%
\[out_i^N(\Gamma): = N_i(\Gamma) = \{ \psi\in \mathrm{Fm} \mid \exists \alpha(\alpha \in \Gamma \ \&\ (\alpha,\psi)\in N_i )\} \] 
 where $N_i\subseteq \mathrm{Fm}\times \mathrm{Fm}$ is the  {\em closure}  of $N$ under (i.e.~the smallest extension of $N$ satisfying)  the  inference rules below, as  specified in the table.
\vspace{1mm}
\begin{center}
\begin{tabular}{lll}
\infer[(\top)]{(\top,\top)}{} \infer[(\bot)]{(\bot,\bot)}{} &
\infer[\mathrm{(SI)}]{(\beta,\varphi)}{(\alpha,\varphi) &  \beta \vdash \alpha} &
\infer[\mathrm{(WO)}]{(\alpha,\psi)}{(\alpha,\varphi) &  \varphi \vdash \psi} \\[2mm]
\infer[\mathrm{(AND)}]{(\alpha,\varphi \land \psi)}{(\alpha,\varphi) & (\alpha,\psi)} &
\infer[\mathrm{(OR)}]{(\alpha \lor \beta, \varphi)}{(\alpha,\varphi) & (\beta,\varphi)} &
\infer[\mathrm{(CT)}]{(\alpha,\psi)}{(\alpha,\varphi) & (\alpha \land \varphi, \psi)}
\end{tabular}
\end{center}

\begin{center}
\label{table1}
	\begin{tabular}{ l l}
		\hline
		$N_i$ & Rules   \\
		\hline	
		$N_{1}$  & $ \mathrm{(\top), (SI), (WO), (AND)}$   \\
		$N_{2}$  & $\mathrm{(\top), (SI), (WO), (AND), (OR)}$  \\
		$N_{3}$  & $\mathrm{(\top), (SI), (WO), (AND), (CT)}$  \\
		$N_{4}$  & $\mathrm{(\top), (SI), (WO),(AND), (OR), (CT)}$ \\
		\hline
  
 \end{tabular}
\captionof{table}{closures of normative systems}
\end{center}      

\begin{remark}
\label{rem:additional rules}
In \cite{OlPaTo23}, the following additional rules are considered:
\begin{center}
\begin{tabular}{ccc}
\infer[\mathrm{(R-AND)}]{(\alpha,\varphi \land \psi)}{(\alpha,\varphi) & (\alpha,\psi) & \alpha\wedge\varphi\wedge\psi\not\vdash \bot}
&&
\infer[\mathrm{(R-CT)}]{(\alpha,\psi)}{(\alpha,\varphi) & (\alpha \land \varphi, \psi)  & \alpha\wedge\varphi\wedge\psi\not\vdash \bot}\\
\infer[\mathrm{(ex-OR)}]{(\alpha \lor \beta, \varphi\lor \psi)}{(\alpha,\varphi) & (\beta,\psi)} &&
\infer[\mathrm{(Eq)}]{(\alpha, \psi)}{(\alpha,\varphi) & \varphi \equiv \psi}\\
\end{tabular}
\end{center}
where $\varphi\equiv\psi$ iff $\varphi\vdash \psi$ and $\psi\vdash\varphi$. The rules in the upper row are versions of $\mathrm{(AND)}$ and $\mathrm{(CT)}$ with a built-in consistency check, while those in the lower row are derivable from $\mathrm{(OR)}$ and  $\mathrm{(WO)}$. These rules give rise to a space of sixteen normative systems, generated by replacing $\mathrm{(AND)}$ and $\mathrm{(CT)}$ (resp.~$\mathrm{(OR)}$ and  $\mathrm{(WO)}$)  in Table \ref{table1} with their modified versions. In the present paper, we only focus on the four types of normative systems indicated in the Table \ref{table1}. However, it is possible to generalize the whole space of normative systems considered in \cite{OlPaTo23} to selfextensional logics, and at the end of the next section we will briefly outline how this can be done.
\end{remark}

\subsection{Permission systems}
\label{ssec:permission systems CPL}


\paragraph{Negative permission systems.} Any  input/output logic $\mathbb{L} = (\mathcal{L}_{\mathrm{CPL}}, N)$  induces the {\em conditional or negative  permission system} $P_N \subseteq \mathrm{Fm}\times \mathrm{Fm}$ (cf.~\cite{Makinson00}) defined as follows: 

\begin{center}
	$P_N: = \{(\alpha,\varphi) \mid (\alpha,\neg \varphi) \notin N\}.$
\end{center}
The same definition applies verbatim to any input/output logic $\mathbb{L} = (\mathcal{L}_{\mathrm{IPL}}, N)$, where $\mathcal{L}_{\mathrm{IPL}}$ denotes intuitionistic propositional logic.

\begin{proposition}
\label{prop: charact neg permission}
    For any input/output logic $\mathbb{L} = (\mathcal{L}_{\mathrm{CPL/IPL}}, N)$ for which  $\mathrm{(WO)}$ holds, $P_N$ is the largest permission system $P$ such that, for all $\alpha,\varphi, \psi\in \mathrm{Fm}$, \begin{center} if $(\alpha, \varphi)\in P$ and $(\alpha, \psi)\in N$, then $Cn(\varphi, \psi)\neq\mathrm{ Fm}$.
\end{center}
\end{proposition}
\begin{proof}
    Let us first show that the property holds for $P_N$. Let $\alpha,\varphi, \psi\in \mathrm{Fm}$ such that   $(\alpha, \psi)\in N$. 
    If $Cn(\varphi, \psi) = \mathrm{Fm}$, then $\psi\vdash_{\mathrm{CPL/IPL}} \neg \varphi$, which would imply, by $\mathrm{(WO)}$, that $(\alpha, \neg \varphi)\in N$, i.e.~$(\alpha, \varphi)\notin P_N$, as required. The same argument shows that any  permission system $P$ for which the property holds must be included in $P_N$.
\end{proof}

\paragraph{Static positive permission systems.} Static positive permission captures the idea that $\psi$ be permitted under $\gamma$ iff it is normatively entailed by some explicitly given permission $(\alpha, \varphi)$ in $P$, given the normative system $N$. In what follows, for any rule $\mathrm{(R)}$, and  any normative system $N$ on $\mathcal{L}_{\mathrm{CPL}}$, we let $N^{(R)}\subseteq \mathrm{Fm}\times \mathrm{Fm}$ denote the closure of $N$ under rule $\mathrm{(R)}$. For any  $(\alpha, \varphi)\in \mathrm{Fm}\times \mathrm{Fm}$, we let $N^{(R)}_{(\alpha, \varphi)}\subseteq \mathrm{Fm}\times \mathrm{Fm}$ denote the  closure of $N\cup\{(\alpha, \varphi)\}$ under rule $\mathrm{(R)}$,  and for any $1\leq i\leq 4$, we let $N^i_{(\alpha, \varphi)}\subseteq \mathrm{Fm}\times \mathrm{Fm}$ denote the closure of $N\cup\{(\alpha, \varphi)\}$ under the rules specified in Table \ref{table1}. 

For any normative system $N$ on $\mathcal{L}_{\mathrm{CPL}}$,  any conditional permission system $P\subseteq P_N$,   
%
and  any rule $\mathrm{(R)}$,  
 the {\em static positive permission systems associated with $N^{(R)}$ and $P$} (cf.~\cite{Makinson03}) 
are defined  as follows:
\begin{center}
	$ S^{\mathrm{(R)}}(P, N) :=\begin{cases} \bigcup\{ N^{\mathrm{(R)}}_{(\alpha, \varphi)} \mid (\alpha, \varphi)\in P\} & \text{if } P\neq \varnothing\\
 N^{\mathrm{(R)}} & \text{ otherwise}.
 \end{cases}$
 \end{center}
For  any $1\leq i\leq 4$,  
 the {\em static positive permission systems associated with $N^i$ and $P$} 
are defined  as follows:
\begin{center}
	$ S^i(P, N) :=\begin{cases} \bigcup\{ N^i_{(\alpha, \varphi)} \mid (\alpha, \varphi)\in P\} & \text{if } P\neq \varnothing\\
 N^i & \text{ otherwise}.
 \end{cases}$

\end{center}
It immediately follows from the definition above that 
\[N^{(R)}\subseteq S^{(R)}(P,N) \text{ and } N^{i}\subseteq S^{i}(P,N).\]
 
\paragraph{Dynamic permission systems.} The  notion of dynamic permission intends to capture the idea that a proposition $\varphi$  be permitted under condition $\alpha$  whenever forbidding it under $\alpha$, given the obligations of the normative system $N$, would entail forbidding some $\psi$ under some satisfiable condition $\gamma$ which is explicitly permitted under $\gamma$. 

\begin{definition} 
\label{def:dynpermclassical}(cf.~\cite{Makinson03})
For any normative system $N$ on $\mathcal{L}_{\mathrm{CPL}}$,  any conditional permission system $P\subseteq P_N$, and any rule $\mathrm{(R)}$,  
 the {\em dynamic positive permission system} $D^{\mathrm{(R)}}(P, N)$ is defined  as follows:

\begin{center}
		$  D^{\mathrm{(R)}}(P, N) = \{ (\alpha, \varphi) \mid \exists \gamma\exists  \psi(\gamma \not\vdash\bot \ \&\ (\gamma,  \psi) \in S^{(R)}(P, N) \ \&\  (\gamma, \neg \psi) \in N^{\mathrm{(R)}}_{(\alpha, \neg \varphi)})\}$,
		
	\end{center}
%
and for  any $1\leq i\leq 3$,  
 the {\em dynamic positive permission system} $D^{i}(P, N)$
is defined  as follows:

\begin{center}
		$  D^{i}(P, N) = \{ (\alpha, \varphi) \mid \exists \gamma\exists  \psi(\gamma \not\vdash\bot \ \&\ (\gamma,  \psi) \in S^{i}(P, N) \ \&\  (\gamma, \neg \psi) \in N^{\mathrm{i}}_{(\alpha, \neg \varphi)})\}$.
		
	\end{center}
 \end{definition}
\section{Normative  systems on selfextensional logics}
\label{sec:normative on selfextensional}
In the present section we build on \cite[Section 2.2]{wollic22}, and introduce generalized versions of normative systems in the framework of selfextensional logics.  

\begin{definition}
\label{def:normative system AAL}
Let $\mathcal{L}= (\mathrm{Fm}, \vdash)$ be a logic in the sense specified in Section \ref{ssec:selfextensional}. 
A  {\em normative system} on $\mathcal{L}$ is a relation $N \subseteq \mathrm{Fm}\times \mathrm{Fm}$, the elements  $(\alpha,\varphi)$ of which are called {\em conditional norms} (or obligations). 

  A normative system  $N \subseteq \mathrm{Fm}\times\mathrm{Fm}$ is {\em internally incoherent} if  $(\alpha, \varphi)$ and $ (\alpha, \psi) \in N$ for some $\alpha, \varphi, \psi\in \mathrm{Fm}$  such that  $Cn(\alpha) \neq \mathrm{Fm}$ and $Cn(\varphi,\psi) = \mathrm{Fm}$; a normative system $N$ is {\em internally coherent} if it is not internally incoherent.
  If  $N, N' \subseteq \mathrm{Fm}\times\mathrm{Fm}$ are  normative systems,  $N$ is {\em almost included} in $N'$ (in symbols: $N\subseteq_c N'$) if $(\alpha,\varphi) \in N$ and $Cn(\alpha)\neq\mathrm{Fm}$ imply $(\alpha,\varphi) \in N'$.
\end{definition}
The intuitive reading of any norm $(\alpha,\varphi)\in N$ remains the same as that discussed in the previous section.
%
For any  $\Gamma\subseteq \mathrm{Fm}$, let $N(\Gamma) := \{ \psi \mid \exists \alpha(\alpha \in \Gamma \ \&\ (\alpha,\psi)\in N )\}$. 
\begin{definition}
\label{def:input output logic}
    An {\em input/output logic} is a tuple $\mathbb{L} = (\mathcal{L}, N)$ s.t.~$\mathcal{L}= (\mathrm{Fm}, \vdash)$ is a (selfextensional) logic, and $N$ is a normative system on $\mathcal{L}$.
\end{definition}

\begin{definition}[Output operations] For any input/output logic $\mathbb{L} = (\mathcal{L}, N)$, and each $1\leq i\leq 4$, the output operation $out_i^N$ is defined as follows: for any $\Gamma\subseteq \mathrm{Fm}$,

%
\[out_i^N(\Gamma): = N_i(\Gamma) = \{ \psi\in \mathrm{Fm} \mid \exists \alpha(\alpha \in \Gamma \ \&\ (\alpha,\psi)\in N_i )\} \] 
 where $N_i\subseteq \mathrm{Fm}\times \mathrm{Fm}$ is the  {\em closure}  of $N$ under (i.e.~the smallest extension of $N$ satisfying)  the  inference rules below, as  specified in the table.

\vspace{1mm}
\begin{center}
\begin{tabular}{llll}
\infer[(\top)]{(\top,\top)}{} \infer[(\bot)]{(\bot,\bot)}{} &
\infer[\mathrm{(SI)}]{(\beta,\varphi)}{(\alpha,\varphi) &  \beta \vdash \alpha} &
\infer[\mathrm{(WO)}]{(\alpha,\psi)}{(\alpha,\varphi) &  \varphi \vdash \psi} \\[2mm]
\infer[\mathrm{(AND)}]{(\alpha,\varphi \land \psi)}{(\alpha,\varphi) & (\alpha,\psi)} &
\infer[\mathrm{(OR)}]{(\alpha \lor \beta, \varphi)}{(\alpha,\varphi) & (\beta,\varphi)} &
\infer[\mathrm{(CT)}]{(\alpha,\psi)}{(\alpha,\varphi) & (\alpha \land \varphi, \psi)} & \infer[\mathrm{(EX)}]{(\alpha , \varphi)}{(\alpha,\varphi\vee \psi) & Cn(\alpha, \psi) = \mathrm{Fm}} 
\end{tabular}
\end{center}

\begin{center}
\label{table2}
	\begin{tabular}{ l l}
		\hline
		$N_i$ & Rules   \\
		\hline	
		$N_{1}$  & $ \mathrm{(\top), (SI), (WO), (AND)}$   \\
		$N_{2}$  & $\mathrm{(\top), (SI), (WO), (AND), (OR)}$  \\
		$N_{3}$  & $\mathrm{(\top), (SI), (WO), (AND), (CT)}$  \\
		$N_{4}$  & $\mathrm{(\top), (SI), (WO),(AND), (OR), (CT)}$ \\
		\hline
  
 \end{tabular}
\captionof{table}{closures of normative systems}
\end{center}      
\end{definition}
Clearly, with the exception of  $\mathrm{(SI)}$ and $\mathrm{(WO)}$,  all the rules above (as well as the rules below and  in the next section) apply only to those input/output logics based on selfextensional logics with the (minimal) metalogical properties guaranteeing the existence of the corresponding term-connectives. So, for instance, rules $\mathrm{(AND)}$ and $\mathrm{(CT)}$ only apply in the context of logics for which $\wedge_P$ holds, and so on. For the sake of a better readability, in the remainder of the paper we will implicitly assume these basic properties, and only mention the additional properties when it is required.

\smallskip
Generalizations of $\mathrm{(AND)}$ and $\mathrm{(OR)}$ which do not require $\wedge_P$ and $\vee_P$, but  are equivalent to these closure rules in the presence of $\mathrm{(SI)}$, $\mathrm{(WO)}$, and  the above mentioned metalogical properties,  are the following:
\begin{center}
    \begin{tabular}{c c}
         \infer[\mathrm{(DD)}]{(\alpha,\chi ) \text{ for some }\chi\in Cn(\varphi, \psi)}{(\alpha,\varphi) & (\alpha,\psi)} &
\infer[\mathrm{(UD)}]{(\gamma, \varphi) \text{ for some } \gamma\in Cn(\alpha)\cap Cn(\beta)}{(\alpha,\varphi) & (\beta, \varphi)}  \\
         
    \end{tabular}
\end{center}
\smallskip
Let us conclude the present section by  discussing the new rule (EX), and how the rules mentioned in Remark \ref{rem:additional rules} can be generalized to selfextensional logics. The rule (EX) describes the natural interaction of normative systems with the coimplication connective $\alpha \pdla \psi$, which in CPL is defined as $\alpha\wedge \neg \psi$. The intuitive meaning of this rule is closely related with the logical principle of  ``modus tollens'': if whenever $\alpha$ either $\varphi$ or $\psi$ should hold, and $\alpha$ logically excludes $\psi$, then $\varphi$ should hold whenever $\alpha$. Notice that any normative system $N$ on CPL for which (WO) holds is also closed under (EX): indeed, if $\alpha, \varphi, \psi\in \mathrm{Fm_{CPL}}$ s.t.~$\alpha N (\varphi\vee \psi)$ and $\vdash_{\mathrm{CPL}}\alpha\wedge \neg \psi$, then $\varphi\vee\psi\vdash_{\mathrm{CPL}}(\varphi\vee\psi)\wedge \top \cong (\varphi\vee\psi)\wedge (\alpha\wedge \neg \psi)\cong (\varphi\wedge (\alpha\wedge \neg \psi))\vee (\psi\wedge (\alpha\wedge \neg \psi))\cong  (\varphi\wedge (\alpha\wedge \neg \psi))\vee \bot\cong \varphi$.

\medskip
While the rules $\mathrm{(ex-OR)}$ and $\mathrm{(Eq)}$ can be considered verbatim (in the case of $\mathrm{(ex-OR)}$, provided of course $\vee_P$ holds for $\mathcal{L}$), the built-in consistency check characterizing $\mathrm{(R-AND)}$ and $\mathrm{(R-CT)}$ can be incorporated as follows:
\begin{center}
\begin{tabular}{ccc}
\infer[\mathrm{(R-AND)}]{(\alpha,\varphi \land \psi)}{(\alpha,\varphi) & (\alpha,\psi) & Cn(\alpha,\varphi,\psi)\neq\mathrm{Fm}}
&&
\infer[\mathrm{(R-CT)}]{(\alpha,\psi)}{(\alpha,\varphi) & (\alpha \land \varphi, \psi)  & Cn(\alpha,\varphi,\psi)\neq\mathrm{Fm}}\\
\end{tabular}
\end{center}

\section{Permission systems on selfextensional logics}
\label{sec:permission systems}

In the present section, we introduce and motivate the extension of  the different notions of permission  studied  in the context of input/output logic \cite{Makinson03}, namely, negative permission (cf.~Section \ref{ssec:negperm}), positive static permission (cf.~Section \ref{ssec:static positive}), and dynamic permission (cf.~Section \ref{ssec:dynamic permission systems}), 
to the  general setting of selfextensional logics. 

\subsection{Negative permission systems}
\label{ssec:negperm}

Since the definition of negative permission as given at the beginning of Section \ref{ssec:permission systems CPL} 
does not apply verbatim to the  environment of normative systems based on the generic logics described in Definition \ref{def:normative system AAL}, 
we will use the characterizing property of $P_N$  stated in Proposition \ref{prop: charact neg permission} for the following definition, which informally says that any $\varphi$   is permitted under a given  $\alpha$ iff  $\varphi$  is not logically inconsistent with any obligation $\psi$ under $\alpha$.   



\begin{definition}
\label{def:compatible P}
    For any input/output logic $\mathbb{L} = (\mathcal{L}, N)$,  
    \begin{center}
        $P_N: = \{(\alpha, \varphi)\mid \forall \psi ((\alpha,\psi)\in N  \Rightarrow Cn(\varphi, \psi) \neq \mathrm{Fm}   )\}$.
    \end{center}
\end{definition}
\begin{proposition}
    For any selfextensional logic $\mathcal{L} = (\mathrm{Fm}, \vdash)$ with $\neg_S$ and $\neg_A$, and any normative system $N$ on $\mathrm{Fm}$ which is closed under $\mathrm{(WO)}$, 
    \[P_N = \{(\alpha, \varphi)\mid (\alpha, \neg\varphi)\notin N\}.\]    \end{proposition}
    \begin{proof}
        For the left-to-right inclusion, let $\psi: = \neg \varphi$; by the assumption and $\neg_A$, $Cn(\varphi, \neg\varphi) = \mathrm{Fm}$ implies that $(\alpha, \neg \varphi)\notin N$, as required.
        Conversely, let $\alpha, \varphi\in \mathrm{Fm}$ s.t.~$(\alpha, \neg \varphi)\notin N$ and let $\psi\in \mathrm{Fm}$ s.t.~$Cn(\varphi, \psi) = \mathrm{Fm}$. By  $\neg_S$, this implies that $\psi\vdash \neg \varphi$. Hence, $(\alpha,  \psi)\notin N$, for otherwise, by $\mathrm{(WO)}$,  $(\alpha, \psi)\in N$ and $\psi\vdash \neg \varphi$ would imply that $(\alpha, \neg\varphi)\in N$, against the assumption.
    \end{proof}
Hence, by the proposition above, Definition \ref{def:compatible P} is equivalent to the definition of negative permission as given at the beginning of Section \ref{ssec:permission systems CPL} in all  settings based on classical and intuitionistic propositional logic, but also on (non distributive) logics such as the logic of pseudocomplemented lattices (cf.~Example \ref{ex:main example}.\ref{item:pseudocomplemented}). 

\begin{proposition}
If $N$ is internally coherent, then $N \subseteq_c P_N$.
\end{proposition}
\begin{proof}
    Let $(\alpha,\varphi) \in N$ with $Cn(\alpha) \neq \mathrm{Fm}$ and suppose for the sake of contradiction that $(\alpha,\varphi) \notin P_N$. Then  a formula $\psi$ exists such that $(\alpha,\psi) \in N$ and $Cn(\varphi,\psi) = \mathrm{Fm}$, making $N$ internally incoherent.
\end{proof}

\begin{proposition} \label{prop:negperm_incl}
    For any $N_1, N_2 \subseteq \mathrm{Fm}$, if $N_1 \subseteq N_2$ then $P_{N_2} \subseteq P_{N_1}$.
\end{proposition}
\begin{proof}
If $(\alpha, \varphi) \notin P_{N_1}$, then $Cn(\varphi,\psi) = \mathrm{Fm}$ and $(\alpha, \psi) \in N_1 \subseteq N_2$ for some $\psi\in \mathrm{Fm}$. Hence,  $(\alpha,\varphi) \notin P_{N_2}$, as required.
\end{proof}

Consider the following closure rules on  $P_N^c: = (\mathrm{Fm}\times \mathrm{Fm})\setminus P_N$:\footnote{For instance, the rule \infer[]{(\alpha,\varphi)}{(\alpha,\psi) &  \varphi \vdash \psi} reads as follows: if $(\alpha, \psi)\notin P_N$ and $\varphi \vdash \psi$ then $(\alpha, \varphi)\notin P_N$.}

\vspace{1mm}

\begin{center}
\begin{tabular}{lll}
\infer[\mathrm{(\top)}^\wt]{(\top,\bot)}{} & \infer[\mathrm{(\bot)}^\wt]{(\bot,\top)}{} &
\infer[\mathrm{(SI)}^\wt]{(\alpha,\varphi)}{(\beta,\varphi) &  \alpha \vdash \beta} \\[2mm]
\infer[\mathrm{(WO)}^\wt]{(\alpha,\varphi)}{(\alpha,\psi) &  \varphi \vdash \psi} &
\infer[\mathrm{(AND)}^\wt]{ (\alpha, \varphi\vee \psi)}{(\alpha,\varphi) & (\alpha, \psi)} &
\infer[\mathrm{(OR)}^\wt]{ (\alpha\vee \beta, \varphi)}{(\alpha,\varphi) & (\beta, \varphi)}
 \\[2mm]
\infer[\mathrm{(CT)}^\wt]{(\alpha, \psi)}{(\alpha,\varphi) \in N& (\alpha \wedge \varphi, \psi)} & \infer[\mathrm{(EX)}^\wt]{(\alpha , \varphi)}{(\alpha,\varphi\rightarrow  \psi)\in N & Cn(\alpha, \psi) = \mathrm{Fm}} 
\end{tabular}
\end{center}
The relationship between each rule $\mathrm{(X)}^\wt$ and its corresponding rule $\mathrm{(X)}$ is similar to the one between the rule scheme $\mathrm{(HR)}$ and $\mathrm{(HR)}^{-1}$ discussed in \cite[Section 2]{Makinson03}.\footnote{The main difference between the study of the properties of conditional permissions in \cite{Makinson03} and  the present study is that the former is developed in terms of the closure properties of $P_N$ itself, whereas  the present one is carried out in terms of the closure properties of the {\em complement} of $P_N$. One reason for this is that $P_N^c$ can be understood as a system of {\em prohibitions}, and hence as a particular type of normative system, studying which in terms of rules formulated as Horn-type conditions seems to provide a greater conceptual uniformity. This uniformity is also reflected in  the properties of the algebraic structures which are used as a semantic environment for permission systems in input/output logic in the companion paper \cite{part2}. We refer to this paper for an expanded  discussion on this issue.} Specifically, each  rule $\mathrm{(X)}^\wt$ has been obtained by reading off the equivalence ``$(\alpha, \varphi)\in N$ iff $(\alpha,\neg\varphi)\in P_N^c$'' from the corresponding rule $\mathrm{(X)}$  and then applying manipulations which yield logically equivalent conditions in classical logic. For instance, as to $\mathrm{(\top)^\wt}$, we proceed as follows: $(\top, \top)\in N$ iff $ (\top, \neg \top)\in P_N^c$ iff $(\top, \bot)\in P_N^c$; as to $\mathrm{(WO)^\wt}$, we rewrite ``$(\alpha, \psi)\in N$ and $\psi\vdash \varphi$ entail $(\alpha, \varphi)\in N$'' as ``$(\alpha, \neg \psi)\in P_N^c$ and $\neg\varphi\vdash \neg\psi$ entail $(\alpha, \neg \varphi)\in P_N^c$'', then we instantiate  $\psi: = \neg\psi$ and $\varphi: = \neg\varphi$ and use the fact that classical negation is involutive.  Of course,  having used classically valid logical equivalences to generate these rules does not imply that these rules are mere reformulations of the closure rules for normative systems in every context. The picture is more nuanced, as  the next proposition shows.
\begin{proposition} \label{prop:negperm_list} For any input/output logic $\mathbb{L} = (\mathcal{L}, N)$,  
	 \begin{enumerate}
	     \item If $\bot_P$ and $\top_W$ hold for $\mathcal{L}$, $P^c_N$ is closed under $(\top)^\wt$ iff $(\top,\psi) \in N$ for some $\psi\in \mathrm{Fm}$.
	     \item If $\bot_P$ and $\top_W$ hold for $\mathcal{L}$ and $N$ is closed under $\mathrm{(WO)}$, then $N$ is closed under $(\top)$ iff $P_N^c$ is closed under $(\top)^\wt$.
	     \item If $\bot_P$ and $\top_W$ hold for $\mathcal{L}$, $P^c_N$ is closed under $(\bot)^\wt$ iff  $(\bot, \psi) \in N$   for some $\psi\in \mathrm{Fm}$ s.t.~$Cn(\psi) = \mathrm{Fm}$.
	     \item
      If $\bot_P$ and $\top_W$ hold for $\mathcal{L}$ and $N$ is closed under $\mathrm{(WO)}$, then $N$ is closed under $(\bot)$ iff $P_N^c$ is closed under $(\bot)^\wt$.
	     \item \label{item:closure under wowt} $P^c_N$ is closed under $\mathrm{(WO)}^\wt$.
	     \item \label{item:closure under siwt}If $N$ is closed under $\mathrm{(SI)}$, then $P^c_N$ is closed under $(\mathrm{SI})^\wt$.
	     \item If   $\bot_P$, $\top_P$, $\neg_I$, $\neg_A$, and $\neg_P$ hold for $\mathcal{L}$ and $N$ is closed under  $(\mathrm{WO})$, then  $N$ is closed under $\mathrm{(SI)}$ iff $P_N^c$ is closed under $\mathrm{(SI)}^\wt$.
	     \item \label{item:closure under andwt}  If  $\wedge_P$ and $\vee_S$ hold for $\mathcal{L}$  and $N$ is closed under $(\mathrm{AND})$, then $P^c_N$ is closed under $(\mathrm{AND})^\wt$.
	     \item If  $\bot_P$, $\top_P$, $\neg_I$, $\neg_A$, and $\neg_P$ hold for $\mathcal{L}$ and $N$ is closed under $\mathrm{(WO)}$, then $N$ is closed under $\mathrm{(AND)}$ iff $P_N^c$ is closed under $\mathrm{(AND)}^\wt$.
	     \item \label{item:closure under orwt} If $\vee_S$ holds for $\mathcal{L}$  and $N$ is closed under $\mathrm{(OR)}$ and $\mathrm{(WO)}$, then $P^c_N$ is closed under $\mathrm{(OR)}^\wt$.
	     \item If $\bot_P$, $\top_P$, $\neg_I$, $\neg_A$, and $\neg_P$ hold for $\mathcal{L}$ and $N$ is closed under  $(\mathrm{WO})$, then  $N$ is closed under $\mathrm{(OR)}$ iff $P_N^c$ is closed under $\mathrm{(OR)}^\wt$.
	     \item \label{item:closure under ctwt}  If $N$ is closed under $\mathrm{(CT)}$, then $P^c_N$ is closed under $\mathrm{(CT)}^\wt$.
	     \item If  $\bot_P$, $\top_P$, $\neg_I$, $\neg_A$, and $\neg_P$ hold for $\mathcal{L}$ and $N$ is closed under $(\mathrm{WO})$, then $N$ is closed under $\mathrm{(CT)}$ iff $P^c_N$ is closed under $\mathrm{(CT)}^\wt$.
	 \end{enumerate}
\end{proposition}
\begin{proof}
    \begin{enumerate}
        \item $(\top,\bot)\in P_{N}^c$ iff $(\top,\bot)\notin P_{N}$ iff $(\top, \psi)\in N$ and $Cn(\bot, \psi) = \mathrm{Fm}$ for some $\psi\in \mathrm{Fm}$, iff $(\top, \psi)\in N$  for some $\psi\in \mathrm{Fm}$, since $\bot_P$ implies that the second conjunct is always true.
        \item By  the previous item, it is enough to show that $(\top, \top)\in N$ iff $(\top, \psi)\in N$ for some $\psi\in \mathrm{Fm}$. This equivalence is guaranteed by the assumption that $N$ be closed under $\mathrm{(WO)}$.
        \item $(\bot, \top)\in P_{N}^c$ iff $(\bot, \top)\notin P_{N}$ iff $(\bot, \psi)\in N$ and $Cn(\top, \psi) = \mathrm{Fm}$ for some $\psi\in \mathrm{Fm}$ iff $(\bot, \psi)\in N$  for some $\psi\in \mathrm{Fm}$ s.t.~$Cn(\psi) = \mathrm{Fm}$. The last equivalence holds because of $\top_W$.
        \item By  the previous item, it is enough to show that $(\bot, \bot)\in N$ iff $(\bot, \psi)\in N$ for some $\psi\in \mathrm{Fm}$ s.t.~$Cn(\psi) = \mathrm{Fm}$. This equivalence is guaranteed by $\psi\vdash \bot$ and the assumption that $N$ be closed under $\mathrm{(WO)}$.
        \item Let $\alpha, \varphi,\psi\in \mathrm{Fm}$ s.t.~$\varphi \vdash \psi$ and $(\alpha,\psi) \notin P_N$. Hence, some  $\psi'\in \mathrm{Fm}$ exists s.t.~$(\alpha, \psi') \in N$ and $Cn(\psi,\psi')=\mathrm{Fm}$. Since $\varphi\vdash \psi$, the latter identity implies that $Cn(\varphi,\psi')=\mathrm{Fm}$.  Thus, $\psi'$ is a witness for $(\alpha,\varphi) \notin P_N$, as required. 
        \item Let $\alpha, \beta,\varphi\in \mathrm{Fm}$ s.t.~$\alpha \vdash \beta$ and $(\beta,\varphi) \notin P_N$. Hence,  $(\beta,\psi) \in N$ and $  Cn(\varphi,\psi) = \mathrm{Fm}$ for some $\psi\in \mathrm{Fm}$. Since $N$ is closed under $\mathrm{(SI)}$, from $\alpha\vdash \beta$ and $(\beta,\psi) \in N$, it follows that $(\alpha, \psi) \in N$. Hence, $\psi$ is also a witness to $(\alpha,\varphi) \notin P_N$, as required.
        \item By the previous item, the proof is complete if we  show that $P_N^c$ being closed under $\mathrm{(SI)}^\wt$ implies that $N$ is closed under $\mathrm{(SI)}$. Let $\alpha, \beta,\varphi\in \mathrm{Fm}$ s.t.~$\alpha \vdash \beta$ and  $(\beta,\varphi) \in N$, and let us show that $(\alpha, \varphi)\in N$. By $\neg_A$ and the definition of $P_N$, from $Cn(\varphi, \neg \varphi) = \mathrm{Fm}$ it follows that $(\beta, \neg \varphi) \notin P_N$, which implies, by $(\mathrm{SI})^\wt$, that $(\alpha,\neg \varphi) \notin P_N$, i.e.~$(\alpha,\varphi') \in N$ and $Cn(\neg \varphi, \varphi') = Fm$  for some $\varphi'\in \mathrm{Fm}$. By Proposition \ref{prop:preliminary_properties}.\ref{prop:strong_negation}, the last identity implies that $\varphi' \vdash \neg \neg \varphi$,  which implies $\varphi' \vdash  \varphi$ by $\neg_I$. Hence,  by $\mathrm{(WO)}$, we conclude $(\alpha,  \varphi) \in N$, as required.

        \item Arguing contrapositively, let $\alpha, \varphi,\psi\in \mathrm{Fm}$ s.t.~$(\alpha,\varphi\vee \psi)\in P_N$. By definition and $\vee_S$, this means that $Cn(\chi, \varphi)\cap Cn(\chi, \psi) = Cn(\chi, \varphi\vee\psi) = \mathrm{Fm}$ for every $\chi\in \mathrm{Fm}$ s.t.~$(\alpha, \chi)\in N$. Hence, for every such $\chi$, either  $Cn(\chi, \varphi) = \mathrm{Fm} = Cn(\chi, \psi)$, which implies that $(\alpha, \varphi)\in P_N$ and  $(\alpha, \psi)\in P_N$, as required.
        
        \item By the previous item, the proof is complete if we  show that $P_N^c$ being closed under $\mathrm{(AND)}^\wt$ implies that $N$ is closed under $\mathrm{(AND)}$. Let $\alpha, \varphi,\psi\in \mathrm{Fm}$ s.t.~$(\alpha,\varphi)\in N$ and $(\alpha,\psi) \in N$. By $\neg_A$ and the definition of $P_N$, from  $Cn(\neg \varphi,\varphi) =\mathrm{Fm} = Cn(\neg \psi,\psi)$ we deduce that $(\alpha, \neg \varphi)\notin P_N$ and $(\alpha, \neg \psi) \notin P_N$, which implies,  by $\mathrm{(AND)}^\wt$, that $ (\alpha, \neg \varphi \vee \neg \psi)\notin P_N$. This entails, since $P_N^c$ is closed under  $(\mathrm{WO})^\wt$ (cf.~item 8) and $\neg(\varphi\wedge \psi)\vdash \neg\varphi\vee \neg\psi $ by Lemma \ref{lemma:antitonicity of neg}, that $(\alpha,\neg (\varphi\wedge\psi)) \notin P_N$. By definition, this means that  $(\alpha, \varphi') \in N$ and $Cn(\neg (\varphi\wedge\psi),\varphi') = \mathrm{Fm}$ for some $\varphi'\in \mathrm{Fm}$. Hence, by Proposition \ref{prop:preliminary_properties}.\ref{prop:strong_negation} and $\neg_I$, $\varphi' \vdash \neg \neg (\varphi\wedge\psi)\vdash\varphi\wedge\psi$. By $\mathrm{(WO)}$, this implies that $(\alpha, \varphi\wedge \psi) \in N$, as required.
        \item Let $\alpha, \beta, \varphi\in \mathrm{Fm}$ s.t.~$(\alpha,\varphi)\notin P_N$ and $ (\beta,\varphi) \notin P_N$. Hence, $(\alpha,\psi_1)\in N$ and $(\beta, \psi_2) \in N$ for some $\psi_1, \psi_2\in \mathrm{Fm}$ s.t.~$Cn(\varphi,\psi_1) = \mathrm{Fm} = Cn(\varphi,\psi_2)$.  By $(\mathrm{WO})$ and $(\mathrm{OR})$ and $\vee_S$, this implies that $(\alpha \vee \beta, \psi_1 \vee \psi_2) \in N$. Moreover, $\vee_S$ entails that $Cn(\varphi, \psi_1\vee\psi_2)= Cn(\varphi, \psi_1)\cap Cn(\varphi, \psi_2)=\mathrm{Fm}\cap\mathrm{Fm}=\mathrm{Fm}$. This shows that $\psi_1\vee\psi_2$ is a witness for $(\alpha\vee \beta, \varphi)\notin P_N$, as required.
        \item By the previous item, the proof is complete if we  show that $P_N^c$ being closed under $\mathrm{(OR)}^\wt$ implies that $N$ is closed under $\mathrm{(OR)}$. Let $\alpha, \beta,\varphi\in \mathrm{Fm}$ s.t.~$(\alpha,\varphi)\in N$ and $(\beta,\varphi) \in N$. By $\neg_A$ and the definition of $P_N$, from  $Cn(\neg \varphi,\varphi) =\mathrm{Fm}$ we deduce that $(\alpha, \neg \varphi)\notin P_N$ and $(\beta, \neg \varphi) \notin P_N$, which implies,  by $\mathrm{(OR)}^\wt$, that $ (\alpha \vee \beta, \neg \varphi )\notin P_N$. By definition, this means that  $(\alpha\vee \beta, \varphi') \in N$ and $Cn(\neg \varphi,\varphi') = \mathrm{Fm}$ for some $\varphi'\in \mathrm{Fm}$. Hence, by Proposition \ref{prop:preliminary_properties}.\ref{prop:strong_negation} and $\neg_I$, $\varphi' \vdash \neg \neg \varphi\vdash\varphi$. By $\mathrm{(WO)}$, this implies that $(\alpha\vee \beta, \varphi) \in N$, as required.
        \item Let $\alpha, \varphi, \psi\in \mathrm{Fm}$ s.t.~$(\alpha, \varphi) \in N$ and $(\alpha \wedge \varphi, \psi) \notin P_N$. Hence, $(\alpha \wedge \varphi, \psi') \in N$ for some $\psi'\in \mathrm{Fm}$ s.t.~$Cn(\psi,\psi') =  \mathrm{Fm}$. By $\mathrm{(CT)}$, this implies that $(\alpha,\psi') \in N$, hence $\psi'$ is the witness for $(\alpha,\psi)\notin P_N$, as required.
        \item By the previous item, the proof is complete if we  show that $P_N^c$ being closed under $\mathrm{(CT)}^\wt$ implies that $N$ is closed under $\mathrm{(CT)}$. Let $\alpha, \varphi,\psi\in \mathrm{Fm}$ s.t.~$(\alpha,\varphi)\in N$ and $(\alpha\wedge \varphi,\psi) \in N$. By $\neg_A$ and the definition of $P_N$, from  $ Cn(\neg \psi,\psi) = \mathrm{Fm}$ we deduce that  $(\alpha\wedge \varphi, \neg \psi) \notin P_N$, which implies,  by $\mathrm{(CT)}^\wt$, that $ (\alpha,  \neg \psi)\notin P_N$. By definition, this means that  $(\alpha, \varphi') \in N$ and $Cn(\neg \psi,\varphi') = \mathrm{Fm}$ for some $\varphi'\in \mathrm{Fm}$. Hence, by Proposition \ref{prop:preliminary_properties}.\ref{prop:strong_negation} and $\neg_I$, $\varphi' \vdash \neg \neg \psi\vdash\psi$. By $\mathrm{(WO)}$, this implies that $(\alpha, \psi) \in N$, as required.

    \end{enumerate}
\end{proof}
For any input/output logic $\mathbb{L} = (\mathcal{L}, N)$ and any $1\leq i\leq 4$, we let $P_i: = P_{N_i}$.
\begin{corollary}
 For any input/output logic $\mathbb{L} = (\mathcal{L}, N)$, if $\wedge_P$, $\vee_S$, $\bot_P$, and $\top_W$ hold for $\mathcal{L}$, then
 $P_i^c$ for $1\leq i\leq 4$ is closed  under the rules indicated in the following table.
\end{corollary}

\begin{center}
	\begin{tabular}{ l l}
		\hline
		$P_i^c$ & Rules   \\
		\hline	
		$P_{1}^c$  & $ \mathrm{(\top)^\wt, (SI)^\wt, (WO)^\wt, (AND)^\wt}$   \\
		$P_{2}^c$  & $\mathrm{ (\top)^\wt, (SI)^\wt, (WO)^\wt, (AND)^\wt, (OR)^\wt}$  \\
		$P_{3}^c$  & $\mathrm{ (\top)^\wt, (SI)^\wt, (WO)^\wt, (AND)^\wt, (CT)^\wt}$  \\
		$P_{4}^c$  & $\mathrm{ (\top)^\wt, (SI)^\wt, (WO)^\wt,(AND)^\wt, (OR)^\wt, (CT)^\wt}$ \\
		\hline
	\end{tabular}
\end{center}

\subsection{Dual negative permission systems} \label{ssec:dual negative}
The  perspective afforded by the general setting of selfextensional logics 
 makes it possible to consider a notion of {\em dual conditional permission system} $D_N$ associated with a given normative system $N$, which, in the setting of classical propositional logic, is absorbed by the usual notion of negative permission:
\[(\alpha, \varphi)\in D_N \text{ iff } (\neg\alpha, \varphi)\notin N \text{ iff } (\neg\alpha, \neg\varphi)\in P_N.\]
 Similarly to the generalized definition of $P_N$ introduced in the previous subsection,  a more general version of $D_N$ can be introduced, namely:

\[D_N: = \{(\alpha, \varphi)\mid \exists\beta((\beta, \varphi)\notin N\ \&\  Cn(\alpha, \beta) = \mathrm{Fm})\},\]
    
\noindent which cannot be subsumed by the definition of $P_N$. While the notion of negative permission $P_N$ intuitively characterizes those states of affair $\alpha$ and $\varphi$ which can both be the case without generating a violation of the  normative system $N$, the dual negative permission system $D_N$ characterizes those states of affair  $\alpha$ and $\varphi$ which can both {\em fail} to be the case without generating a violation of the  normative system $N$. 

\begin{proposition}
    For any selfextensional logic $\mathcal{L} = (\mathrm{Fm}, \vdash)$ with $\neg_S$ and $\neg_A$, and any normative system $N$ on $\mathrm{Fm}$ which is closed under $\mathrm{(SI)}$, 
    \[D_N = \{(\alpha, \varphi)\mid (\neg\alpha, \varphi)\notin N\}.\]    \end{proposition}
    \begin{proof}
        For the right-to-left inclusion, take $\beta: = \neg \alpha$ as the witness; by $\neg_A$, we have $Cn(\alpha, \neg\alpha) = \mathrm{Fm}$, as required.
        Conversely,  let $\alpha, \varphi\in \mathrm{Fm}$ s.t.~$(\beta, \varphi)\notin N$ for some $\beta\in \mathrm{Fm}$ s.t.~$Cn(\alpha, \beta) = \mathrm{Fm}$. By  $\neg_S$, this implies that $\beta\vdash \neg \alpha$. Hence, $(\neg\alpha,  \varphi)\notin N$, for otherwise, by $\mathrm{(SI)}$,  $(\beta, \varphi)\in N$, against the assumption.
    \end{proof}

We introduce the following closure rules  on  $D_N^c: = (\mathrm{Fm}\times \mathrm{Fm})\setminus D_N$:

\vspace{1mm}

\begin{center}
\begin{tabular}{lll}
\infer[\mathrm{(\top)}^\tw]{(\bot,\top)}{} & \infer[\mathrm{(\bot)}^\tw]{(\top,\bot)}{} &
\infer[\mathrm{(SI)}^\tw]{(\beta,\varphi)}{(\alpha,\varphi) &  \alpha \vdash \beta} \\[2mm]
\infer[\mathrm{(WO)}^\tw]{(\alpha,\psi)}{(\alpha,\varphi) &  \varphi \vdash \psi} &
\infer[\mathrm{(AND)}^\tw]{ (\alpha, \varphi\wedge \psi)}{(\alpha,\varphi) & (\alpha, \psi)} &
\infer[\mathrm{(OR)}^\tw]{ (\alpha\wedge \beta, \varphi)}{(\alpha,\varphi) & (\beta, \varphi)}
 \\[2mm]
$\infer[\mathrm{(CT)}^\tw]{(\alpha, \psi)}{( \alpha, \varphi)& (\varphi \pdla\alpha, \psi)\in N}$
&
\end{tabular}
\end{center}

\begin{proposition}
\label{prop:dual negperm_list}
For any input/output logic $\mathbb{L} = (\mathcal{L}, N)$, 
\begin{enumerate}
    \item If $\bot_P$ and $\top_W$ hold for $\mathcal{L}$, then $D^c_N$ is closed under $(\mathrm{\top})^\tw$ iff  $(\beta, \top)\in N$ for all $\beta \in \mathrm{Fm}$.
    \item If $\bot_P$ and $\top_W$ hold for $\mathcal{L}$ and $N$ is closed under $\mathrm{(SI)}$, then $D^c_N$ is closed under $(\mathrm{\top})^\tw$ iff $N$ is closed under $(\top)$.
     \item If $\bot_P$ and $\top_W$ hold for $\mathcal{L}$, $D^c_N$ is closed under $(\bot)^\tw$ iff  for all $\beta \in \mathrm{Fm}$, if $Cn(\beta) = \mathrm{Fm}$ then $(\beta, \bot) \in N$.
    \item If $\bot_P$ and $\top_W$ hold for $\mathcal{L}$ and $N$ is closed under $\mathrm{(SI)}$, then $D^c_N$ is closed under $(\mathrm{\bot})^\tw$ iff $N$ is closed under $\mathrm{(\bot)}$. 
    \item  $D^c_N$ is closed under $(\mathrm{SI})^\tw$.
    \item If $N$ is closed under $\mathrm{(WO)}$, then $D^c_N$ is closed under $(\mathrm{WO})^\tw$.
    \item If $\bot_P$, $\top_P$, $\neg_I$, $\neg_A$, and $\neg_P$ hold for $\mathcal{L}$, and $N$ is closed under $\mathrm{(SI)}$, then $N$ is closed under $\mathrm{(WO)}$ iff $D^c_N$ is closed under $(\mathrm{WO})^\tw$.
    \item If $\wedge_P$ holds for $\mathcal{L}$ and $N$ is closed under $\mathrm{(AND)}$, then $D^c_N$ is closed under $(\mathrm{AND})^\tw$.
    \item If $\wedge_P$, $\bot_P$, $\top_P$, $\neg_I$, $\neg_A$, and $\neg_P$ hold for $\mathcal{L}$, and $N$ is closed under $\mathrm{(SI)}$, then $N$ is closed under $\mathrm{(AND)}$ iff $D^c_N$ is closed under $(\mathrm{AND})^\tw$.
    \item If $\wedge_P$, $\vee_P$, $\bot_P$, $\top_P$, $\neg_I$, $\neg_A$, and $\neg_P$ hold for $\mathcal{L}$, and $N$ is closed under $\mathrm{(SI)}$ and  $\mathrm{(OR)}$, then $D^c_N$ is closed under $(\mathrm{OR})^\tw$.
    \item If $\wedge_P$, $\vee_S$, $\bot_P$, $\top_P$, $\neg_I$, $\neg_A$, and $\neg_P$ hold for $\mathcal{L}$, and $N$ is closed under $\mathrm{(SI)}$, then $N$ is closed under $\mathrm{(OR)}$ iff $D^c_N$ is closed under $(\mathrm{OR})^\tw$.
    \item If  $\vee_P$ and $\pdla_P$ hold for $\mathcal{L}$, and $N$ is closed under $\mathrm{(WO)}$, $\mathrm{(SI)}$, and $\mathrm{(EX)}$, then $N$ being closed  under $\mathrm{(CT)}$ implies $D^c_N$ being closed under $(\mathrm{CT})^\tw$. 
    \item If $\pdla_P$, $\wedge_S$, $\vee_P$,  $\top_P$  hold for $\mathcal{L}$ and $\neg_A$, $\neg_S$, $\neg_{Ir}$ and $\sim_A$ hold for $\mathcal{L}$ relative to the same term, and $N$ is closed under $\mathrm{(SI)}$, then $D^c_N$ being closed under $(\mathrm{CT})^\tw$ implies that $N$ is closed under $\mathrm{(CT)}$.
\end{enumerate}
\end{proposition}

\begin{proof}
    \begin{enumerate}
        \item $(\bot,\top)\in D^c_N$ iff $(\beta, \top)\in N$ or $Cn(\bot, \beta)\neq \mathrm{Fm}$ for all $\beta \in \mathrm{Fm}$ iff $(\beta, \top)\in N$ for all $\beta \in \mathrm{Fm}$, since $\bot_P$ implies that the second disjunct is always false. 
        \item By the previous item it is enough to show that $(\top, \top)\in N$  implies $(\beta, \top)\in N$ for all $\beta \in \mathrm{Fm}$. This
equivalence is guaranteed by the assumption that $N$ be closed under (SI).
        \item $(\top,\bot) \in D^c_N$ iff for all $\beta \in \mathrm{Fm}$, if $Cn(\top,\beta) = \mathrm{Fm}$ then $(\beta, \bot) \in N$ iff for all $\beta \in \mathrm{Fm}$, if $Cn(\beta) = \mathrm{Fm}$ then $(\beta, \bot) \in N$. The last equivalence is guaranteed by $\top_W$. 
        \item By the previous item, $(\top,\bot) \in D^c_N$ iff for all $\beta \in \mathrm{Fm}$, if $Cn(\beta) = \mathrm{Fm}$ then $(\beta, \bot) \in N$, iff $(\bot, \bot)\in N$. The last equivalence is guaranteed by $\beta\vdash \bot$ and the assumption that $N$ be closed under (SI).  
        
        \item  Arguing contrapositively, let $\alpha, \beta, \varphi \in \mathrm{Fm}$ s.t.~$\alpha \vdash \beta$ and $(\beta,\varphi) \in D_N$. Hence, some  $\alpha'\in \mathrm{Fm}$ exists s.t.~$(\alpha', \varphi) \notin N$ and $Cn(\beta,\alpha')=\mathrm{Fm}$. Since $\alpha\vdash \beta$, the latter identity implies that $Cn(\alpha,\alpha')=\mathrm{Fm}$.  Thus, $\alpha'$ is a witness for $(\alpha,\varphi) \in D_N$, as required.


        \item   Let $\alpha, \varphi, \psi \in \mathrm{Fm}$ s.t.~$(\alpha,\varphi)\notin D_N$ and $\varphi\vdash\psi$. Hence, for any $\beta \in \mathrm{Fm}$, if $Cn(\alpha, \beta)  = \mathrm{Fm}$, then  $(\beta, \varphi) \in N$. Since $N$ is closed under  $\mathrm{(WO)}$, 
 from $\varphi \vdash \psi$ it follows that $(\beta, \psi) \in N$. This shows that $(\alpha, \psi) \notin D_N$, as required.
        
        \item   By the previous item, the proof is complete if we  show that $D_N^c$ being closed under $\mathrm{(WO)}^\tw$ implies that $N$ is closed under $\mathrm{(WO)}$. Arguing contrapositively, let $\alpha, \varphi, \psi \in \mathrm{Fm}$ s.t.~$(\alpha, \psi) \notin N$ and $\varphi \vdash \psi$. Hence, by $\neg_A$ and the definition of $D_N$, from $Cn(\alpha, \neg \alpha) = \mathrm{Fm}$ it follows that  $(\neg \alpha, \psi) \in D_N$, which implies, by $\mathrm{(WO)}^\tw$, that $(\neg \alpha, \varphi) \in D_N$, i.e.~$(\beta, \varphi) \notin N$ and $Cn(\neg \alpha, \beta) = \mathrm{Fm}$ for some $\beta \in \mathrm{Fm}$. By Proposition \ref{prop:preliminary_properties}.\ref{prop:strong_negation}, the last identity implies $\beta \vdash \neg \neg \alpha$, which implies $\beta \vdash \alpha$ by $\neg_I$. Hence, by $\mathrm{(SI)}$ we conclude $(\alpha, \varphi) \notin N$, as required.
        
        \item    Let $\alpha, \varphi,\psi \in \mathrm{Fm}$ s.t.~$(\alpha,\varphi)\notin D_N$ and $(\alpha,\psi)\notin D_N$. Hence, by definition,  $(\beta, \varphi)\in N$ and $(\beta,\psi) \in N$ for every $\beta \in \mathrm{Fm}$ s.t.~$Cn(\alpha, \beta) = \mathrm{Fm}$. By $\mathrm{(AND)}$, this implies that $(\alpha, \varphi \wedge \psi) \notin D_N$, as required.

        \item   By the previous item, the proof is complete if we  show that $D_N^c$ being closed under $\mathrm{(AND)}^\tw$ implies that $N$ is closed under $\mathrm{(AND)}$. Arguing contrapositively, let $\alpha, \varphi, \psi \in \mathrm{Fm}$ s.t.~$(\alpha, \varphi \wedge \psi) \notin N$. By $\neg_A$ and the definition of $D_N$, this implies that $(\neg \alpha,\varphi \wedge \psi) \in D_N$, which implies, by $\mathrm{(AND)}^\tw$, that either $(\neg \alpha, \varphi) \in D_N$ or $(\neg \alpha, \psi) \in D_N$. Without loss of generality, let  $(\neg \alpha, \varphi) \in D_N$, i.e.~$Cn(\neg \alpha, \beta) = \mathrm{Fm}$ and $(\beta, \varphi) \notin N$ for some $\beta \in \mathrm{Fm}$.  Proposition \ref{prop:preliminary_properties}.\ref{prop:strong_negation} and $\neg_I$ imply that $\beta \vdash \neg \neg \alpha\vdash \alpha$. By $\mathrm{(SI)}$, this implies that $(\alpha, \varphi) \notin N$, as required. 
        
        \item   
        Arguing contrapositively, let $\alpha, \beta,\varphi \in \mathrm{Fm}$ s.t.~$(\alpha\wedge\beta,\varphi)\in D_N$, i.e.~$Cn(\alpha\wedge\beta, \gamma)=\mathrm{Fm}$ and $(\gamma, \varphi)\notin N$ for some $\gamma \in \mathrm{Fm}$.  Proposition \ref{prop:preliminary_properties}.\ref{prop:strong_negation} and Lemma \ref{lemma:antitonicity of neg}.2 imply that $\gamma\vdash \neg\alpha\vee\neg\beta$. By $\mathrm{(SI)}$, this implies  that $(\neg\alpha\vee\neg\beta, \varphi)\notin N$. By $\mathrm{(OR)}$, the last statement implies that $(\neg\alpha, \varphi)\notin N$  or  $(\neg\beta, \varphi)\notin N$, which implies, by $\neg_A$, that $(\alpha, \varphi)\in D_N$ or $(\beta, \varphi)\in D_N$, as required.

\item  By the previous item, the proof is complete if we  show that $D_N^c$ being closed under $\mathrm{(OR)}^\tw$ implies that $N$ is closed under $\mathrm{(OR)}$. Arguing contrapositively, let $\alpha, \beta, \varphi \in \mathrm{Fm}$ such that $(\alpha \vee \beta, \varphi ) \notin N$. By $\neg_A$ and the definition of $D_N$, this implies $(\neg(\alpha\vee\beta), \varphi) \in D_N$, and since $\neg\alpha\wedge\neg\beta\vdash\neg(\alpha\vee\beta)$ holds (cf.~Lemma \ref{lemma:antitonicity of neg second}), by  $\mathrm{(SI)}^\tw$, this implies $(\neg\alpha \wedge \neg\beta, \varphi)\in D_N$. Applying  $\mathrm{(OR)}^\tw$ contrapositively, it follows that either $(\neg \alpha, \varphi) \in D_N$ or $(\neg \beta, \varphi) \in D_N$. Without loss of generality, let $(\neg \alpha, \varphi)\in D_N$, i.e.~$Cn(\neg \alpha, \gamma) = \mathrm{Fm}$ and $(\gamma, \varphi) \notin N$ for some $\gamma \in \mathrm{Fm}$. By Proposition \ref{prop:preliminary_properties}.\ref{prop:strong_negation} and $\neg_I$, it follows that $\gamma \vdash \neg \neg \alpha\vdash \alpha$. Hence, by $\mathrm{(SI)}$, we conclude $(\alpha, \varphi) \notin N$, as required. 
        
    \item 
Let $\alpha, \varphi, \psi \in \mathrm{Fm}$ s.t.~$(\alpha, \psi) \in D_N$ and   $\varphi\pdla \alpha\prec \psi$. Hence, $\beta\not\prec\psi$ for some $\beta\in \mathrm{Fm}$ 
 s.t.~$Cn(\alpha,\beta) = \mathrm{Fm}$. To show that  $(\alpha, \varphi)\in D_N$, it is enough to show that $\beta\not\prec\varphi$. Indeed, the assumption that $\vee_P$ and $\pdla_P$ hold for $\mathcal{L}$ and $\varphi\pdla\alpha\vdash \varphi\pdla\alpha$ imply, by Proposition \ref{prop:preliminary_properties}.8, that $\varphi\vdash (\varphi\pdla \alpha)\vee \alpha$. Hence, by (WO),  $\beta\prec\varphi\vdash (\varphi\pdla \alpha)\vee \alpha$ implies that $\beta\prec  (\varphi\pdla \alpha)\vee \alpha$. This implies, since $Cn(\alpha,\beta) = \mathrm{Fm}$ and (EX), that $\beta\prec  (\varphi\pdla \alpha)$.
By (SI), $\beta\wedge (\varphi\pdla \alpha)\vdash (\varphi\pdla \alpha)\prec \psi$ implies that $\beta\wedge (\varphi\pdla \alpha)\prec \psi$, hence by (CT), we conclude that $\beta\prec\psi$, against the assumption. 

    \item 

     Let $\alpha, \varphi, \psi \in \mathrm{Fm}$ s.t.~$(\alpha, \psi) \notin N$, and let us show that either $\alpha\wedge \varphi\not\prec \psi$ or $(\alpha,\varphi)\notin N$. By $\neg_A$, the assumption that $(\alpha, \psi) \notin N$ implies that  $(\neg \alpha, \psi) \in D_N$, which implies, by  $\mathrm{(CT)}^\tw$, that either $(\varphi\pdla \neg \alpha, \psi) \notin N$ or $(\neg\alpha,\varphi) \in D_N$. Let us assume that $(\varphi\pdla \neg \alpha, \psi) \notin N$, and  let us first show that $\varphi\pdla \neg \alpha \vdash \alpha\wedge \varphi$;   by Proposition \ref{prop:preliminary_properties}.8 and $\pdla_P$, this is true iff $\varphi\vdash (\alpha\wedge\varphi)\vee \neg\alpha$, which is true, since, by $\wedge_S$,  $\top_P$ and $\sim_{A}$, the following chain of identities holds: $(\alpha\wedge\varphi)\vee \neg\alpha = (\alpha\vee \neg\alpha)\wedge (\varphi\vee \neg\alpha) = \top\wedge  (\varphi\vee \neg\alpha) = \varphi\vee \neg\alpha$. Hence, by (SI), $(\varphi\pdla \neg \alpha, \psi) \notin N$ and $\varphi\pdla \neg \alpha \vdash \alpha\wedge \varphi$ imply $\alpha\wedge\varphi\not\prec \psi$, as required. 
     Finally,  $(\neg\alpha,\varphi) \in D_N$ iff $\beta\not\prec \varphi$ for some $\beta\in\mathrm{Fm}$ s.t.~$Cn(\beta, \neg\alpha) = \mathrm{Fm}$. This implies, by $\neg_S$ and $\neg_{Ir}$, that $\beta\vdash \neg\neg\alpha\vdash \alpha$; hence, by (SI), we conclude $(\alpha, \varphi)\notin N$, as required.
    
    \end{enumerate}
\end{proof}

For any input/output logic $\mathbb{L} = (\mathcal{L}, N)$ and any $1\leq i\leq 4$, we let $D_i: = D_{N_i}$.

\begin{corollary}
 For any input/output logic $\mathbb{L} = (\mathcal{L}, N)$, if $\wedge_P$, $\vee_S$, $\top_P$, $\bot_P$, and $\neg_W$ hold for $\mathcal{L}$, then
 $D_i^c$ for $1\leq i\leq 4$ is closed  under the rules indicated in the following table.
\end{corollary}

\begin{center}
	\begin{tabular}{ l l}
		\hline
		$D_i^c$ & Rules   \\
		\hline	
		$D_{1}^c$  & $ \mathrm{(\top)^\tw, (SI)^\tw, (WO)^\tw, (AND)^\tw}$   \\
		$D_{2}^c$  & $\mathrm{ (\top)^\tw, (SI)^\tw, (WO)^\tw, (AND)^\tw, (OR)^\tw}$  \\
		$D_{3}^c$  & $\mathrm{ (\top)^\tw, (SI)^\tw, (WO)^\tw, (AND)^\tw, (CT)^\tw}$  \\
		$D_{4}^c$  & $\mathrm{ (\top)^\tw, (SI)^\tw, (WO)^\tw,(AND)^\tw, (OR)^\tw, (CT)^\tw}$ \\
		\hline
	\end{tabular}
\end{center}

\subsection{Static  positive permission systems}
\label{ssec:static positive}

The definition of static positive permission,  originally introduced in the setting of CPL (cf.~Section \ref{ssec:permission systems CPL}), can be generalized verbatim to the context of any selfextensional logic (cf.~Section  \ref{ssec:selfextensional}), as is done in the following 
\begin{definition} 
For any normative system $N$ on a selfextensional logic $\mathcal{L}$,  any conditional permission system $P\subseteq P_N$,   
%
and  any rule $\mathrm{(R)}$,  
 the {\em static positive permission systems associated with $N^{(R)}$ and $P$} 
are defined  as follows:
\begin{center}
	$ S^{\mathrm{(R)}}(P, N) :=\begin{cases} \bigcup\{ N^{\mathrm{(R)}}_{(\alpha, \varphi)} \mid (\alpha, \varphi)\in P\} & \text{if } P\neq \varnothing\\
 N^{\mathrm{(R)}} & \text{ otherwise}.
 \end{cases}$
 \end{center}
For  any $1\leq i\leq 4$,  
 the {\em static positive permission systems associated with $N^i$ and $P$} 
are defined  as follows:
\begin{center}
	$ S^i(P, N) :=\begin{cases} \bigcup\{ N^i_{(\alpha, \varphi)} \mid (\alpha, \varphi)\in P\} & \text{if } P\neq \varnothing\\
 N^i & \text{ otherwise}.
 \end{cases}$

\end{center}
\end{definition}
 In what follows, we will suppress the index in the notation of positive permission whenever properties considered in each context do not depend on the specific closure properties.
\begin{definition}
\label{def:cross-coherent}
   Let $\mathcal{L} = (\mathrm{Fm}, \vdash)$ be a selfextensional logic. A normative system  $N$ and a permission system  $P$ on $\mathcal{L}$ are {\em cross-incoherent} if $(\gamma,\varphi) \in N$ and $(\gamma, \psi) \in S(P,N)$ for some $\gamma,\varphi,\psi\in \mathrm{Fm}$  s.t.~$Cn(\gamma) \neq \mathrm{Fm}$ and $ Cn(\varphi,\psi) = \mathrm{Fm}$. If $P,N$ are not cross-incoherent, we say they are {\em cross-coherent}.
\end{definition}

\begin{proposition}
    For every $P,N \subseteq\mathrm{Fm}\times\mathrm{Fm}$, $S(P,N) \subseteq_c P_N$ if and only if $P,N$ are cross-coherent.
\end{proposition}
\begin{proof}
    For the left-to-right direction, if $P$ and $N$ are cross-incoherent, i.e.~some $\gamma$, $\varphi$ and $\psi\in \mathrm{Fm}$ exist such that $(\gamma,\varphi) \in N$, $(\gamma, \psi) \in S(P,N)$ and $Cn(\gamma) \neq \mathrm{Fm} = Cn(\varphi,\psi)$, it is easy to see that $(\gamma, \varphi) \notin P_N$, taking $\psi$ as the witness, which shows that $S(P,N) \not\subseteq_c P_N$, as required.

    Conversely, assume $Cn(\alpha) \neq \mathrm{Fm}$ and $(\alpha, \varphi) \in S(P,N)$. If  $(\alpha, \varphi) \notin P_N$, then  $Cn(\varphi, \psi) = \mathrm{Fm}$ and $(\alpha, \psi) \in N$ for some formula $\psi$, contradicting the cross-coherence of $P$ and $N$. 
\end{proof}
Consider the following closure rules: 
\begin{center}
\begin{tabular}{lll}
\infer[\mathrm{(AND)}^\downarrow]{(\alpha, \varphi \wedge \psi)}{(\alpha, \varphi)\in N &(\alpha, \psi) } &
\infer[\mathrm{(OR)}^\downarrow]{(\alpha \vee \beta,\varphi)}{(\alpha,  \varphi)\in N&(\beta, \varphi)} & \infer[\mathrm{(CT)}^\downarrow]{(\alpha, \varphi \wedge \psi)}{(\alpha, \varphi)\in N& (\alpha \wedge \varphi,  \psi)}
\end{tabular}
\end{center}
\begin{proposition}
\label{prop:properties positive}
For any input/output logic $\mathbb{L} = (\mathcal{L}, N)$   any conditional permission system $P\subseteq P_N$,  and any rule $(\mathrm{X})\in \{(\top), \mathrm{(SI)}, \mathrm{(WO)}\}$,
\begin{enumerate}
    \item 
    $S^{\mathrm{(AND)}}(P, N)$ is closed under $\mathrm{(AND)}^\downarrow$.
    \item 
    $S^{\mathrm{(OR)}}(P, N)$ is closed under $\mathrm{(OR)}^\downarrow$.
    \item if $N^{\mathrm{(CT)}}$ and 
$N^{\mathrm{(CT)}}_{(\beta, \gamma)}$ are closed under $\mathrm{(AND)}$ for any $(\beta, \gamma)\in P$,  
    then $S^{\mathrm{(CT)}}(P, N)$ is closed under $\mathrm{(CT)}^\downarrow$.
    \item   $S^{\mathrm{(X)}}(P, N)$ is closed under $\mathrm{(X)}$.
    \item if $S^i(P,N)$ is closed under $\mathrm{(SI)}$ and $\mathrm{(CT)}^\downarrow$, then $S^i(P,N)$ is closed under $\mathrm{(AND)}^\downarrow$.
\end{enumerate}
\end{proposition}
\begin{proof}
1. Let $\alpha, \varphi, \psi\in \mathrm{Fm}$ s.t.~$(\alpha, \varphi)\in N$ and $(\alpha, \psi)\in S^{\mathrm{(AND)}}(P, N)$. If $P=\varnothing$, then $S^{\mathrm{(AND)}}(P, N) = N^{\mathrm{(AND)}}$, which is closed under $\mathrm{(AND)}$.  Hence, $(\alpha, \varphi\wedge\psi)\in  N^{\mathrm{(AND)}}$, as required. If $P \not= \varnothing$, then by definition,   $(\alpha, \psi)\in S^{\mathrm{(AND)}}(P, N)$ implies that $ (\alpha, \psi)\in N_{(\beta, \gamma)}^{\mathrm{(AND)}}$ for some  $(\beta, \gamma)\in P$. Since  $N_{(\beta, \gamma)} ^{\mathrm{(AND)}}$ is closed under $\mathrm{(AND)}$, it follows that $(\alpha, \varphi \wedge \psi)\in N_{(\beta, \gamma)}^{\mathrm{(AND)}}\subseteq S^{\mathrm{(AND)}}(P, N)$, as required. 

2. Let $\alpha, \beta, \varphi\in \mathrm{Fm}$ s.t.~$(\alpha, \varphi)\in N$ and $(\beta, \varphi)\in S^{\mathrm{(OR)}}(P, N)$. If $P=\varnothing$, then $S^{\mathrm{(OR)}}(P, N) = N^{\mathrm{(OR)}}$, which is closed under $\mathrm{(OR)}$.  Hence, $(\alpha\vee \beta, \varphi)\in  N^{\mathrm{(OR)}}$, as required. If $P \not= \varnothing$, then by definition,   $(\beta, \varphi)\in S^{\mathrm{(OR)}}(P, N)$ implies that $ (\beta, \varphi)\in N_{(\beta', \gamma)}^{\mathrm{(OR)}}$ for some  $(\beta', \gamma)\in P$. Since  $N_{(\beta', \gamma)} ^{\mathrm{(OR)}}$ is closed under $\mathrm{(OR)}$, it follows that $(\alpha\vee \beta, \varphi)\in N_{(\beta', \gamma)}^{\mathrm{(OR)}}\subseteq S^{\mathrm{(OR)}}(P, N)$, as required. 

3. Let $\alpha, \varphi, \psi\in \mathrm{Fm}$ s.t.~$(\alpha, \varphi)\in N$ and $(\alpha\wedge \varphi, \psi)\in S^{\mathrm{(CT)}}(P, N)$. If $P=\varnothing$, then $S^{\mathrm{(CT)}}(P, N) = N^{\mathrm{(CT)}}$.  Since $N^{\mathrm{(CT)}}$ is closed under $\mathrm{(CT)}$, we conclude that $(\alpha, \psi)\in N^{\mathrm{(CT)}}$, and from $N^{\mathrm{(CT)}}$ being closed under $\mathrm{(AND)}$ we conclude that  $(\alpha, \varphi\wedge \psi)\in N^{\mathrm{(CT)}}$, as required. If $P \not= \varnothing$, then, by definition, $(\alpha\wedge\varphi, \psi)\in S^{\mathrm{(CT)}}(P, N)$ iff $ (\alpha\wedge\varphi, \psi)\in N_{(\beta, \gamma)}^{\mathrm{(CT)}}$ for some  $(\beta, \gamma)\in P$. Since  $N_{(\beta, \gamma)} ^{\mathrm{(CT)}}$ is closed under $\mathrm{(CT)}$, it follows that $(\alpha, \psi)\in N^{\mathrm{(CT)}}_{(\beta, \gamma)}$. Moreover, since $N^{\mathrm{(CT)}}_{(\beta, \gamma)}$ is closed under $\mathrm{(AND)}$, it follows that   $(\alpha, \varphi\wedge\psi)\in N^{\mathrm{(CT)}}_{(\beta, \gamma)}\subseteq S^{\mathrm{(CT)}}(P, N)$, as required.

4. Immediately follows from $N^{\mathrm{(X)}}$ and $N_{(\alpha, \varphi)}^{\mathrm{(X)}}$ being closed under $\mathrm{(X)}$.

5. Let $(\alpha, \varphi) \in N$ and $(\alpha, \psi) \in S^i(P,N)$. Since $S^i(P,N)$ is closed under $\mathrm{(SI)}$ we have $(\alpha \wedge \varphi, \psi) \in S^i(P,N)$ and applying $\mathrm{(CT)}^\downarrow$ we get $(\alpha, \varphi \wedge \psi) \in S^i(P,N)$.
\end{proof}

\begin{corollary}
For any input/output logic $\mathbb{L} = (\mathcal{L}, N)$, any permission system $P\subseteq P_N$ and all $1\leq i\leq 3$, the  static positive permission system $S^{i}(P, N)$  is closed  under the rules indicated in the following table:

\begin{center}
	\begin{tabular}{ l l}
		\hline
		$S^{i}(P, N)$ & Rules   \\
		\hline	
		$S^{1}(P, N)$  & $ \mathrm{(\top), (SI), (WO), (AND)^\downarrow}$   \\
		$S^{2}(P, N)$  & $\mathrm{ (\top), (SI), (WO), (AND)^\downarrow, (OR)^\downarrow}$  \\
		$S^{3}(P, N)$  & $\mathrm{ (\top), (SI), (WO), (AND)^\downarrow, (CT)^\downarrow}$  \\
		\hline
	\end{tabular}
\end{center}   
\end{corollary}
Notice that $(\mathrm{SI})$ and $ (\mathrm{CT})^\downarrow$ imply $(\mathrm{AND})^\downarrow$, hence the mention of $(\mathrm{AND})^\downarrow$ is redundant in $S^{3}(P, N)$.
\subsection{Dynamic permission systems}
\label{ssec:dynamic permission systems}

In the present section,  we explore some possible generalizations of the definition  originally introduced in \cite{Makinson03} (cf.~Section \ref{ssec:permission systems CPL}) to various input/output settings based on selfextensional logics.
 
Unlike the case of the static permission, the definition of dynamic permission cannot be applied verbatim to the generic setting of arbitrary selfextensional logics. This motivates the following

\begin{definition} 
\label{def:dynpermselfext}
For any normative system $N$ on any selfextensional logic $\mathcal{L} = (\mathrm{Fm}, \vdash)$,  any conditional permission system $P\subseteq \mathrm{Fm}\times\mathrm{Fm}$, and any rule $\mathrm{(R)}$, the {\em dynamic positive permission system} $D^{\mathrm{(R)}}(P, N)\subseteq \mathrm{Fm}\times\mathrm{Fm}$ 
is defined  as follows:

\begin{center}
\begin{tabular}{ccl}
		$  D^{\mathrm{(R)}}(P, N)$ & $ =$ &$ \{ (\alpha, \varphi) \mid \exists \gamma \exists \psi \exists \varphi'\exists \psi' (Cn(\gamma) \neq \mathrm{Fm} \ \&\ (\gamma,  \psi) \in S^{\mathrm{(R)}}(P, N) $\\
 && $  \ \&\ Cn(\psi, \psi') = \mathrm{Fm} = Cn(\varphi, \varphi') \ \&\  (\gamma, \psi') \in N^{\mathrm{(R)}}_{(\alpha, \varphi')})\}$,
 \end{tabular}
		
	\end{center}

%
and for any $1\leq i\leq 3$,  
 the {\em dynamic positive permission system} $D^{i}(P, N)\subseteq \mathrm{Fm}\times\mathrm{Fm}$ 
is defined  as follows:

\begin{center}
\begin{tabular}{ccl}
		$  D^{i}(P, N)$ & $ =$ &$ \{ (\alpha, \varphi) \mid \exists \gamma \exists \psi \exists \varphi'\exists \psi' (Cn(\gamma) \neq \mathrm{Fm}\ \&\ (\gamma,  \psi) \in S^{i}(P, N) $\\
 && $  \ \&\ Cn(\psi, \psi') = \mathrm{Fm} = Cn(\varphi, \varphi') \ \&\  (\gamma, \psi') \in N^i_{(\alpha, \varphi')})\}$.
 \end{tabular}
		
	\end{center}
 \end{definition}
 Much in the same spirit of \cite[Definition 6.1]{deon2020}, the definition above aims at maintaining the intended meaning of the original definition while abstracting away from the specific signature of a given selfextensional logic: indeed, it says that $(\alpha, \varphi)$ is dynamically permitted if an explicit permission $(\gamma, \psi)$ exists, with $\gamma$ consistent, together with formulas $\varphi'$ and $ \psi'$ which are logically inconsistent with $\varphi$ and $\psi$ respectively, such that including $(\alpha, \varphi')$ as a norm would entail admitting a norm $(\gamma, \psi')$ which is inconsistent with the explicit permission $(\gamma, \psi)$. In this definition,   consistency and inconsistency have been expressed purely at the level of the closure operator induced by the consequence relation of the given selfextensional logic.

 Notice that, if $\neg_A$ holds, then $\neg\varphi$ and $\neg\psi$ serve as  {\em canonical witnesses} for the roles of $\varphi'$ and $\psi'$; 
 hence, Definition \ref{def:dynpermclassical} (and hence \cite[Definition 6.1]{deon2020}) implies  Definition \ref{def:dynpermselfext}; however, even in the presence of $\neg_I$, $\neg_A$, $\neg_P$, and  $\mathrm{(WO)}$, Definition \ref{def:dynpermselfext} does not imply Definition \ref{def:dynpermclassical}, since from the assumptions one gets $(\gamma, \neg\psi)\in N^{(\mathrm{R})}_{(\alpha, \varphi')}\supseteq N^{(\mathrm{R})}_{(\alpha, \neg \varphi)}$, while the latter inclusion can be proper, and hence it is not difficult to find counterexamples to the converse implication. 
In what follows, we sometimes omit the superscripts when the statements do not depend on the closure we take.
\begin{lemma}
For any input/output logic $\mathbb{L} = (\mathcal{L}, N)$ s.t.~property $\neg_A$ holds for $\mathcal{L}$, any rule $\mathrm{(R)}$,  and any permission system $P\subseteq \mathrm{Fm}\times\mathrm{Fm}$,
  \[S^{\mathrm{(R)}}(P,N)\subseteq_c D^{\mathrm{(R)}}(P,N).\]  
\end{lemma}
\begin{proof}
Let $(\alpha, \varphi)\in S^{\mathrm{(R)}}(P,N)$ s.t.~$Cn(\alpha)\neq \mathrm{Fm}$. Then the statement is verified letting $\gamma \coloneqq \alpha$, $\psi \coloneqq \varphi$, $\varphi' \coloneqq \neg \varphi$, and $\psi' \coloneqq \neg \varphi$ in the definition of $D^{\mathrm{(R)}}(P,N)$.
\end{proof}

 \begin{proposition}
 \label{prop:first_equiv_dyn}
     For any input/output logic $\mathbb{L} = (\mathcal{L}, N)$ and any permission system $P$ on $\mathcal{L}$,
     \begin{enumerate}
     \item $(\alpha,\varphi) \in D(P,N)$ iff $N_{(\alpha,\varphi')}$ and $P$ are cross-incoherent (cf.~Definition \ref{def:cross-coherent}) for some $\varphi'\in \mathrm{Fm}$ s.t.~$Cn(\varphi,\varphi') = \mathrm{Fm}$.
     \item If $P,N$ are cross-coherent, then 
     \[\bigcap\{P_H\ |\ N\subseteq H \text{ and } H \text{ and } P \text{ cross-coherent}\} \subseteq D(P,N).\]
     \end{enumerate}
 \end{proposition}
 \begin{proof}
 
    1. By definition, $Cn(\varphi,\varphi') = \mathrm{Fm}$ and $N_{(\alpha, \varphi')}$ and $P$ are cross-incoherent iff $(\gamma,\psi) \in S(P,N)$ and $(\gamma,\psi') \in N_{(\alpha,\varphi')}$ for some  $\gamma$, $\psi$ and $\psi'\in \mathrm{Fm}$ with $Cn(\gamma) \neq \mathrm{Fm} = Cn(\psi,\psi')$. This is exactly what $(\alpha,\varphi)\in D(P,N)$ means.

    2. By the previous item, it is enough to show that if $N_{(\alpha,\varphi')}$ and $P$ are cross-coherent, then $(\alpha, \varphi) \notin P_{N_{(\alpha, \varphi')}}$, where $\varphi'$ is such that $Cn(\varphi,\varphi') = \mathrm{Fm}$. Given that $(\alpha,\varphi') \in N_{(\alpha,\varphi')}$, then assuming $(\alpha, \varphi) \in P_{N_{(\alpha, \varphi')}}$ would imply $Cn(\varphi, \varphi') \neq \mathrm{Fm}$, contrary to our assumptions. Hence, $(\alpha, \varphi) \notin P_{N_{(\alpha, \varphi')}}$.
\end{proof}
The proposition above motivates the following  definition of generalized dynamic permission system:
\begin{definition} \label{def:dynpermE}
    For any normative system $N$ on any selfextensional logic $\mathcal{L} = (\mathrm{Fm}, \vdash)$,  any conditional permission system $P\subseteq \mathrm{Fm}\times\mathrm{Fm}$, any rule $\mathrm{(R)}$, and any nonempty up-directed set $\mathcal{N}^{(R)}_N\subseteq \mathcal{P}(\mathrm{Fm}\times\mathrm{Fm})$ ($\mathcal{P}$ being the powerset operator) such that every element $H$ of $\mathcal{N}^{(R)}_N$ is closed under $\mathrm{(R)}$, cross-coherent with $P$, and such that $N\subseteq H$, the {\em dynamic positive permission system} $E^{\mathrm{(R)}}(P, N, \mathcal{N})$ 
is defined  as follows:

\[E^{(R)}(P,N,\mathcal{N}) = \bigcap\{P_H\ |\ H \in \mathcal{N}^{(R)}_N\},\]

and for any $1\leq i\leq 3$,  
 the {\em generalized dynamic positive permission system} $E_{i}(P, N,\mathcal{N})$ 
is defined  as follows:

\[E_{i}(P,N,\mathcal{N}) = \bigcap\{P_H\ |\ H \in \mathcal{N}^{i}_N\}.\]
\end{definition}
Informally, the set $\mathcal{N}_N^{\mathrm{(R)}}$ represents a given space of possible $\mathrm{(R)}$-closed expansions of the normative system $N$  which are cross-coherent with $P$. The  condition that $\mathcal{N}_N^{\mathrm{(R)}}$ be up-directed corresponds to the requirement that the order in which new norms are added does not affect the result. 
\begin{proposition} 
\label{prop:generalized dynamic perm}
	 For any input/output logic $\mathbb{L} = (\mathcal{L}, N)$, any $P\subseteq \mathrm{Fm}\times\mathrm{Fm}$, any rule $\mathrm{(X) \in \{(\top), (\bot),}$ $\mathrm{  (SI), (WO), (CT)}\}$, and any $\mathcal{N}^{\mathrm{\mathrm{(X)}}}_N$ as in Definition \ref{def:dynpermE}, let  $ (E^{\mathrm{(X)}}(P, N,\mathcal{N}))^c: = (\mathrm{Fm}\times \mathrm{Fm})\setminus E^{\mathrm{(X)}}(P, N,\mathcal{N})$,

  \begin{enumerate}
      \item $(E^{\mathrm{(X)}}(P, N,\mathcal{N}))^c$ is closed under $\mathrm{(X)}^{\wt}$. 
      \item If $\mathcal{L}$ satisfies $\wedge_P$ and $\vee_S$, then $(E^{\mathrm{(AND)}}(P, N,\mathcal{N}))^c$ is closed under $\mathrm{(AND)}^{\wt}$.
      \item If $\mathcal{L}$ satisfies $\vee_S$, then $(E^{\mathrm{(WO),(OR)}}(P, N,\mathcal{N}))^c$ is closed under $\mathrm{(OR)}^{\wt}$.
  \end{enumerate}
\end{proposition}

\begin{proof}
     1. Since $\mathcal{N}^{\mathrm{(X)}}_N$ is not empty, we trivially have $E^{\mathrm{(X)}}(P,N,\mathcal{N}) \subseteq P_{N^{\mathrm{(X)}}}$, that is, $P^c_{N^{\mathrm{(X)}}} \subseteq (E^{\mathrm{(X)}}(P,N,\mathcal{N}))^c$. Hence, the statement holds for $\mathrm{\mathrm{(X)}} \in \{(\bot),(\top)\}$.
    To prove the statement for $\mathrm{(X)}: = \mathrm{(SI)}$, let $\alpha, \beta,\varphi\in \mathrm{Fm}$ s.t.~$\alpha\vdash\beta$ and $(\beta, \varphi) \notin E^{\mathrm{(SI)}}(P,N,\mathcal{N})$, i.e.~$(\beta,\varphi) \notin P_H$ for some $H$ in $\mathcal{N}^{\mathrm{(SI)}}_N$. By  Proposition \ref{prop:negperm_list}.\ref{item:closure under siwt}, $P^c_H$ is closed under $\mathrm{(SI)}^\wt$, hence  $(\alpha, \varphi) \notin P_H$. This implies that $(\alpha, \varphi) \notin E^{\mathrm{(SI)}}(P,N,\mathcal{N})$, as required. The cases in which $\mathrm{(X)}: = \mathrm{(WO)}$ and  $\mathrm{(X)}: =\mathrm{(CT)}$  are proven similarly, using Proposition \ref{prop:negperm_list}.\ref{item:closure under wowt}  and \ref{prop:negperm_list}.\ref{item:closure under ctwt}. 
    
    2. Let $\alpha, \varphi, \psi\in \mathrm{Fm}$ s.t.~$(\alpha,\varphi)$ and $(\alpha,\psi) \notin E^{\mathrm{(AND)}}(P,N,\mathcal{N})$, i.e.~$(\alpha,\varphi) \notin P_{H_1}$ and $(\alpha,\psi) \notin P_{H_2}$ for some $H_1,H_2 \in \mathcal{N}^{\mathrm{(AND)}}_N$. Since $\mathcal{N}^{\mathrm{(AND)}}_N$ is up-directed, some $H \in \mathcal{N}^{\mathrm{(AND)}}_N$ exists s.t.~$H_1,H_2\subseteq H$. By  Proposition \ref{prop:negperm_incl} we have $(\alpha,\varphi),(\alpha,\psi) \notin P_H$, which implies that $(\alpha,\varphi \vee\psi) \notin P_H$, since $P_H^c$ is closed under $\mathrm{(AND)^\wt}$ by Proposition \ref{prop:negperm_list}.\ref{item:closure under andwt}. This shows that $(\alpha, \varphi\vee\psi) \notin E^{\mathrm{(AND)}}(P,N,\mathcal{N})$, as required.

    3. Similar to the previous item, using Proposition \ref{prop:negperm_list}.\ref{item:closure under orwt}.
\end{proof}

\begin{corollary}
For any input/output logic $\mathbb{L} = (\mathcal{L}, N)$ such that $\wedge_P$ and $\vee_S$ hold\footnote{In fact, the properties of $E_1$ do not require $\vee_S$.} for $\mathcal{L}$, any $P\subseteq \mathrm{Fm}\times\mathrm{Fm}$ and any $1\leq i\leq 3$, the relative set-theoretic complement of $E^{i}(P, N,\mathcal{N})$  is closed under the rules indicated in the following table:

\begin{center}
	\begin{tabular}{ l l}
		\hline
		$(E_{i}(P, N))^c$ & Rules   \\
		\hline	
		$E_{ 1}^c$  & $ \mathrm{(\top)^\wt, (SI)^\wt, (WO)^\wt, (AND)^\wt}$   \\
		$E_{ 2}^c$  & $\mathrm{ (\top)^\wt, (SI)^\wt, (WO)^\wt, (AND)^\wt, (OR)^\wt}$  \\
		$E_{ 3}^c$  & $\mathrm{ (\top)^\wt, (SI)^\wt, (WO)^\wt, (CT)^\wt}$  \\
		\hline
	\end{tabular}
\end{center}  
\end{corollary}

\section{Conclusions}
\label{sec:conclusions}
\paragraph{Results of the present paper.} The present paper further develops the line of research  initiated in  \cite{wollic22}, where normative systems  on selfextensional logics have been introduced and studied from a semantic perspective in connection with subordination algebras \cite{subordination}. In the present paper, the framework of normative systems on selfextensional logics is extended to various notions of permission systems, namely negative, dual negative, static positive, and dynamic positive permission systems, and  their associated closure properties  are studied in connection with the metalogical properties of selfextensional logics. 
\paragraph{Additional rules.} In the present paper, we have focused our attention on the best known closure rules of normative systems, and their direct counterparts applied to the relative complements of permission systems. Moreover, in  Remark \ref{rem:additional rules} and at the end of Section \ref{sec:normative on selfextensional}, we have briefly mentioned weaker variations of these rules and discussed the possibility of generalizing their study to the setting of selfextensional logics. A natural direction is to systematically explore these closure rules both from a syntactic and a semantic perspective.  
\paragraph{Characterizations of output operators.}
In \cite{Makinson00}, various output operators associated with normative systems are characterized both in terms of their being closed under syntactic rules, and in terms of various set-theoretic constructions.  In the present paper, various similar characterizations or sufficient  conditions are introduced (cf.~Propositions \ref{prop:negperm_list}, \ref{prop:dual negperm_list}, \ref{prop:properties positive}, \ref{prop:generalized dynamic perm}) for static positive permission systems and for the relative set-theoretic complements of negative permission systems, positive dynamic permission systems, and the newly introduced notion of dual negative permission systems. These characterizations or sufficient conditions are formulated in terms of closure under syntactic rules.  
In the companion paper \cite{part2}, this syntax-driven approach is complemented by the  semantic approach described in the next paragraph.
\paragraph{Modal characterization of syntactic rules.} 
In \cite{wollic22}, the study of the properties of normative systems in connection with their semantic interpretation on subordination algebras led to their correspondence-theoretic (cf.~\cite{de2021slanted}) characterization in terms of the algebraic validity of modal axioms encoding properties of their associated output operators. The results in \cite{wollic22} cover a finite number of  conditions which reflect well known closure properties of normative systems. A natural direction is to generalize these results to infinite syntactic classes of closure properties. This is the focus of the companion paper \cite{part3}, currently in preparation.

\bibliographystyle{abbrv}
\bibliography{obligation}

\end{document}